\documentclass[11pt, reqno]{amsart}

\usepackage[section]{colepaperlibertine}
\usepackage[french, english]{babel}
\usepackage{scalerel}

\newcommand{\hp}{\mathbb{H}}
\newcommand{\sq}{{\scaleobj{0.7}{\square}}}

\makeatletter
\newenvironment{abstracts}{%
  \ifx\maketitle\relax
    \ClassWarning{\@classname}{Abstract should precede
      \protect\maketitle\space in AMS document classes; reported}%
  \fi
  \global\setbox\abstractbox=\vtop \bgroup
    \normalfont\Small
    \list{}{\labelwidth\z@
      \leftmargin3pc \rightmargin\leftmargin
      \listparindent\normalparindent \itemindent\z@
      \parsep\z@ \@plus\p@
      
      \itemsep\medskipamount
    }%
}{%
  \endlist\egroup
  \ifx\@setabstract\relax \@setabstracta \fi
}

\newcommand{\abstractin}[1]{%
  \otherlanguage{#1}%
  \item[\hskip\labelsep\scshape\abstractname.]%
}
\makeatother

\title{Reaction-diffusion equations in the half-space}

\author{Henri Berestycki}
\address{HB: Centre d'analyse et de math\'{e}matique sociales, EHESS-CNRS, 54 Boulevard Raspail, 75006, Paris, France}
\address{\hspace{1.8em}Senior Visiting Fellow, HKUST Jockey Club Institute for Advanced Study, Hong Kong University of Science and Technology}
\email{\tt hb@ehess.fr}

\author{Cole Graham}
\address{CG: Department of Mathematics, Stanford University, 450 Jane Stanford Way, Building 380, Stanford, CA 94305, USA}
\email{\tt grahamca@stanford.edu}

\begin{document}

\begin{abstracts}
  \abstractin{english}
  We study reaction-diffusion equations of various types in the half-space.
  For bistable reactions with Dirichlet boundary conditions, we prove conditional uniqueness: there is a unique nonzero bounded steady state which exceeds the bistable threshold on large balls.
  Moreover, solutions starting from sufficiently large initial data converge to this steady state as $t \to \infty$.
  For compactly supported initial data, the asymptotic speed of this propagation agrees with the unique speed $c_*$ of the one-dimensional traveling wave.
  We furthermore construct a traveling wave in the half-plane of speed $c_*$.
  
  In parallel, we show analogous results for ignition reactions under both Dirichlet and Robin boundary conditions.
  Using our ignition construction, we obtain stronger results for monostable reactions with the same boundary conditions.
  For such reactions, we show in general that there is a unique nonzero bounded steady state.
  Furthermore, monostable reactions exhibit the hair-trigger effect: every solution with nontrivial initial data converges to this steady state as $t \to \infty$.
  Given compactly supported initial data, this disturbance propagates at a speed $c_*$ equal to the minimal speed of one-dimensional traveling waves.
  We also construct monostable traveling waves in the Dirichlet or Robin half-plane with any speed $c \geq c_*$.

  \bigskip

  {\abstractin{french}
  Nous \'etudions les \'equations de r\'eaction-diffusion de diff{\'e}rents types dans un demi-espace avec conditions au bord de type Dirichlet ou Robin.
  Pour les r{\'e}actions bistables avec conditions de Dirichlet, nous {\'e}tablissons l'unicit{\'e} de la solution stationnaire sup{\'e}rieure {\`a} un certain seuil sur des boules suffisamment  grande.
  Les solutions qui {\'e}manent de donn{\'e}es initiales suffisamment grandes convergent vers cette solution lorsque $t\to\infty$.
  Dans ce cas, la vitesse asymptotique de propagation pour des donn{\'e}es initiales {\`a} support compact est donn{\'e}e par l'unique vitesse $c_*$ des fronts plans.
  De plus, nous construisons un front progressif avec la vitesse $c_*$ dans le demi-plan avec condition de Dirichlet.
  
  En parall{\`e}le, nous montrons des r{\'e}sultats analogues concernant les r{\'e}actions de type ignition pour les conditions de Dirichlet ou de Robin.
  En utilisant cette derni{\`e}re construction, nous obtenons des r{\'e}sultats plus pr{\'e}cis dans le cas de r{\'e}actions monostables avec ces conditions aux limites.
  Nous montrons de fa\c{c}on g{\'e}n{\'e}rale qu'il existe  un unique {\'e}tat stationnaire et que  toutes les solutions avec donn{\'e}es initiales born{\'e}es et non identiquement nulles convergent vers cet {\'e}tat lorsque $t \to \infty$.
  Lorsque la donn\'ee initiale est \`a support compact, nous obtenons une vitesse asymptotique de propagation $c_*$ {\'e}gale \`a la vitesse minimale des fronts progressifs plans.
  De plus, nous construisons des ondes progressives  dans le demi-espace avec conditions de Dirichlet ou de Robin au bord pour toute vitesse $c \geq c_*$.}
\end{abstracts}



\maketitle

\thispagestyle{empty}

\clearpage
\setcounter{page}{1}

\section{Introduction}

We are interested in the long-time behavior of reaction-diffusion equations in the half-space.
We work in $d + 1$ spatial dimensions with $d \geq 1$.
We denote positions in $\R^{d + 1}$ by ${\tbf{x} = (\tbf{x}', y) \in \R^d \times \R}$, and define the upper half-space ${\hp \coloneqq \R^d \times \R_+}$.
We study solutions $u \colon [0,\infty) \times \bar \hp \to \R$ to the following reaction-diffusion equation with Dirichlet ($\varrho = 0$) or Robin ($\varrho > 0$) boundary conditions:
\begin{equation}
  \label{eq:main}
  \begin{cases}
    \partial_t u = \Delta u + f(u) & \textrm{in } \hp,\\
    \partial_y u = \varrho^{-1} u & \textrm{on } \partial \hp.
  \end{cases}
\end{equation}
When $\varrho = 0$, we interpret the boundary condition as $u = \varrho \partial_y u = 0$ on $\partial \hp$.
The boundary $\partial \hp$ makes \eqref{eq:main} anisotropic, in contrast to the equation in the whole space.
We view this asymmetry as a form of inhomogeneity.

The nonlinearity $f$ in \eqref{eq:main} is known as the ``reaction.''
In this work, we consider three classical reaction types: monostable, ignition, and bistable.
Before defining these classes, we discuss the relationship between \eqref{eq:main} and prior work.

Reaction-diffusion equations are widely used to model the spread of a population in an environment.
We decompose this spread into three interlocking phenomena: \emph{invasion}, \emph{propagation}, and \emph{traveling waves}.
Invasion refers to the qualitative behavior of the solution as $t \to \infty$: does the population eventually inhabit its entire environment?
Such ecological dominance is not guaranteed---it depends on the reaction and initial condition.
When a population does invade, we are interested in quantitative aspects of its propagation.
How large of a region does the population occupy at a particular time?
On the whole space, an invading solution $u$ propagates asymptotically linearly in time.
That is, the level sets of $u(t, \anon)$ expand in space at a nearly constant rate, known as the \emph{asymptotic speed of propagation}.
Since the propagation eventually approaches this constant speed, we also search for \emph{traveling waves}: solutions which move at precisely constant speed.
These three spreading phenomena are well understood in homogeneous media in the whole space, and they guide our study of \eqref{eq:main}.

Invasion, propagation, and traveling waves were first systematically studied in pioneering works of Aronson and Weinberger \cite{AW} and Fife and McLeod \cite{FM} in the homogeneous setting.
These fundamental results inspired a vast literature, to which we cannot do justice.
We instead highlight a selection of works; for a wider view of the field, we direct the reader to the references therein.

Aronson and Weinberger proved the \emph{hair-trigger effect} for monostable reactions: nontrivial initial data \emph{always} invade the whole space.
Moreover, all nontrivial solutions with localized initial data eventually propagate at a common asymptotic speed.
The set of monostable reactions includes a special subclass, the so-called Fisher--KPP reactions, which are particularly amenable to linearized analysis.
When $f$ is Fisher--KPP, Bramson \cite{Bramson78, Bramson83} used probabilistic techniques to determine the position of level sets of solutions with great precision; for further results in this direction, see also \cite{Gartner, EvS, NRR, BBD, Graham}.

Ignition and bistable reactions behave differently.
In these cases, ${f \leq 0}$ when $u$ is small.
It follows that the population will go extinct as $t \to \infty$ if $u_0$ is sufficiently small.
On the other hand, sufficiently large initial data \emph{do} invade.
The precise nature of the threshold between extinction and invasion was a longstanding problem, first resolved by Zlato\v{s}~\cite{Zlatos1} for square initial data.
This result has been extended to wider classes of reactions and initial data by Du and Matano~\cite{MD} and Matano and Pol\'{a}\v{c}ik \mbox{\cite{MP1, MP2}}.

We emphasize that these results all hold in the homogeneous setting.
However, applications clearly motivate the study of \emph{inhomogeneous} media.
The most immediate model of inhomogeneity is a spatially dependent evolution equation.
Although pure traveling waves do not exist in typical inhomogeneous media, periodic equations admit generalizations known as \emph{pulsating fronts}.
In the whole space, Freidlin and G\"{a}rtner \cite{FG, Freidlin} and Hamel and the first author \cite{BH} have studied invasion, propagation, and pulsating fronts in periodic media; for more refined results in the Fisher--KPP case, see Hamel \emph{et al.} \cite{HNRR} and Shabani \cite{Shabani}.
In the aperiodic setting, traveling waves must be further generalized to \emph{transition fronts}: entire solutions which asymptotically resemble traveling waves.
For a variety of results on the existence of transition fronts, see Mellet, Roquejoffre, and Sire \cite{MRS} and works of Nolen, Roquejoffre, Ryzhik, and Zlato\v{s} \cite{NR, NRRZ, Zlatos2}.

There is a second important approach to inhomogeneity: we can work in a general domain rather than the whole space.
For instance, an impermeable inclusion in a material can be represented by a domain with a Neumann boundary.
Hamel, Matano, Weinberger, and the first author have investigated propagation and pulsating fronts in periodic domains \cite{Matano, BH, Weinberger, BHM}.
For Fisher--KPP reactions, Hamel, Nadirashvili, and the first author have characterized the spreading speed in both periodic and more general domains \cite{BHN1, BHN2, BHN3}.

Invasion can be a delicate matter in general domains.
For instance, bistable reactions exhibit a phenomenon known as \emph{blocking}: solutions may propagate initially, only to become obstructed by certain geometries.
Bouhours, Chapuisat and the first author \cite{BBC} and Ducasse and Rossi \cite{DR} have studied blocking in channels and periodic domains, respectively.
In the opposite direction, Rossi has recently established the hair-trigger effect for monostable reactions in quite general domains \cite{Rossi}.

Most of the above works confront a common difficulty: their systems vary along the direction of propagation.
To isolate the effects of boundary, it is helpful to remove this complication.
A significant body of work studies reaction-diffusion equations in \emph{cylinders} $\R \times \Omega$ with compact cross-sections $\Omega \subset \R^d$.
Then the problem is translation-invariant in the first coordinate, and solutions only propagate in this direction.
In fact, the equation itself can depend on the transverse coordinates without complicating the analysis.
Nirenberg and the first author \cite{BN}, Mallordy and Roquejoffre \cite{MR, Roquejoffre}, and Muratov and Novaga \cite{Muratov, MN} have all considered traveling waves and propagation in such cylindrical problems.

One can view the present work as an extension of these results to a cylinder with non-compact cross-section.
Indeed, \eqref{eq:main} is invariant under translations parallel to the boundary $\partial \hp$, and our domain may be viewed as the cylinder $\R \times (\R^{d - 1} \times \R_+)$.
Our problem thus  combines the challenges of the cylindrical and multivariate free settings: inhomogeneity and transverse non-compactness.
We study the simplest example with both features: the half-space with a homogeneous equation.
The lack of transverse compactness greatly complicates our analysis of propagation and our construction of traveling waves.
In this sense, our approach to the former has much in common with a recent work of Lou and Lu, who study invasion and propagation for certain Fisher--KPP reactions in cones with Dirichlet conditions \cite{LL}.
We discuss their work in greater detail after Theorem~\ref{thm:ASP} below.

As mentioned above, it is common to work with a Neumann boundary.
However, in the half-space, Neumann conditions reduce to the homogeneous problem.
Indeed, they are equivalent to a free evolution in the whole space that is even in one coordinate.
Here, we consider Dirichlet and Robin conditions.
The boundary thus absorbs mass, and may be viewed as a hostile inhomogeneity which destroys a fraction of the population upon contact.
Much less is understood about the effects of such absorbing boundary conditions.

In our study of the half-space, we are further motivated by ``road-field'' models, which include more general interactions between populations in a half-plane and on a line.
These systems were introduced by Roquejoffre, Rossi, and the first author in \mbox{\cite{BRR1, BRR2, BRR3}}.
They describe individuals moving back and forth between a two-dimensional ``field'' and its one-dimensional boundary, the ``road.''
We can interpret \eqref{eq:main} as a degenerate case of this model, in which individuals that hop on the road never leave it.
With the feedback between road and field broken, we are free to solely consider the population in the field, which suffers steady attrition at the boundary.

We now precisely define the monostable, ignition, and bistable reaction classes.
We always assume that the reaction $f$ is continuous and piecewise $\m C^1$.
In addition, our monostable reactions satisfy the following hypotheses:
\begin{enumerate}[label = \textup{(M\arabic*)}, leftmargin = 5em, labelsep = 0.6em, itemsep= 1ex, topsep = 1ex]
\item
  \label{hyp:mono}
  $f(0) = f(1) = 0$ \hspace{1ex} and \hspace{1ex} $f|_{(0,1)} > 0$;

\item
  \label{hyp:mono-derivs}
  $f'(0^+) > 0$ \hspace{1ex} and \hspace{1ex} $f'(1^-) < 0$.
\end{enumerate}
Ignition reactions obey:
\begin{enumerate}[label = \textup{(I\arabic*)}, leftmargin = 5em, labelsep = 0.6em, itemsep= 1ex, topsep = 1ex]
\item
  \label{hyp:ign}
  $f|_{[0, \theta] \cup \{1\}} \equiv 0$ \hspace{1ex} and \hspace{1ex} $f|_{(\theta, 1)} > 0$ for some $\theta \in (0, 1)$;

\item
  \label{hyp:ign-derivs}
  $f'(\theta^+) > 0$ \hspace{1ex} and \hspace{1ex} $f'(1^-) < 0$.
\end{enumerate}
Finally, bistable reactions satisfy:
\begin{enumerate}[label = \textup{(B\arabic*)}, leftmargin = 5em, labelsep = 0.6em, itemsep= 1ex, topsep = 1ex]
\item
  \label{hyp:bi}
  $f(0) = f(1) = 0,$ \hspace{1ex} $f|_{(0, \theta)} < 0$, \hspace{1ex} and \hspace{1ex} $f|_{(\theta, 1)} > 0$ for some $\theta \in (0, 1)$;

\item
  \label{hyp:bi-derivs}
  $f'(0^+) < 0$ \hspace{1ex} and \hspace{1ex} $f'(1^-) < 0$.
\end{enumerate}
Additionally, we will always assume:
\begin{enumerate}[label = \textup{(B3)}, leftmargin = 5em, labelsep = 0.6em, itemsep= 1ex, topsep = 1ex]
\item
  \label{hyp:bi-pos}
  $\displaystyle \int_0^1 f(r) \d r > 0.$
\end{enumerate}
That is, the state $1$ is ``more stable'' than the state $0$.
This ensures that the one-dimensional wave-speed of $f$ is positive.
This assumption is not part of the traditional definition of bistability, but for simplicity we always use ``bistable'' to mean \ref{hyp:bi}--\ref{hyp:bi-pos}.

Our endpoint assumptions on $f'$ can likely be relaxed somewhat, but we do not pursue the matter here.
For the sake of clarity, we extend $f$ by zero on $\R \setminus [0, 1]$.

We now consider the phenomenon of invasion in \eqref{eq:main}.
We typically work with compactly supported initial data $u(0, \anon) = u_0$ satisfying $0 \leq u_0 \leq 1$ and ${u_0 \not \equiv 0}$.
In the remainder of the paper, we write these two conditions as $0 \lneqq u_0 \leq 1$.
In the whole space, monostable reactions exhibit the hair-trigger effect \cite{AW}: any solution with initial data $0 \lneqq u_0 \leq 1$ converges locally uniformly to $1$, the stable zero of $f$, as $t \to \infty$.
In contrast, ignition and bistable reactions cause solutions with small $u_0$ to converge uniformly to $0$.
Nonetheless, sufficiently large $u_0$ still invade in these cases \cite{Kanel2, AW, FM}.

We prove the analogue of these results in $\hp$.
However, the constant function $1$ does not satisfy our boundary conditions on $\partial \hp$.
Rather, when solutions invade, we expect them to converge to a nonconstant steady state $\varphi$ in $\hp$ which is independent of $\tbf{x}'$.
That is, $\varphi = \varphi(y)$ should satisfy
\begin{equation}
  \label{eq:steady-ODE}
  \varphi'' + f(\varphi) = 0 \And \varphi'(0) = \varrho^{-1} \varphi(0).
\end{equation}
We show that this ODE has a unique nonzero bounded solution if $f$ is monostable or ignition.
However, when $f$ is bistable, uniqueness is only guaranteed under Dirichlet boundary conditions.
For this reason, we confine our study of bistable reactions to the Dirichlet case.
The long-time behavior of \eqref{eq:main} for $f$ bistable and $\varrho > 0$ remains an interesting open question.

When $f$ is monostable, we show that $\varphi$ is also the unique nonzero bounded steady state in $\hp$.
This uniqueness is less clear for ignition and bistable reactions, but $\varphi$ is the only bounded steady state which exceeds $\theta$ on large balls.
\begin{theorem}
  \label{thm:steady}\
  \begin{enumerate}[label = \textup{(\Alph*)}, leftmargin = 3em, labelsep = 0.6em, itemsep= 1ex, topsep = 1ex]
  \item
    \label{item:steady-mono}
    Let $f$ be monostable with $\varrho \in [0, \infty)$.
    Then $\varphi = \varphi(y)$ is the unique nonzero bounded steady state of \eqref{eq:main}.

  \item
    \label{item:steady-ign-bi}
    Let $f$ be ignition with $\varrho \in [0, \infty)$ or bistable with $\varrho = 0$.
    Then for all ${\delta \in (0, 1 - \theta)}$, there exists $R_{\trm{steady}}(\delta) > 0$ such that $\varphi = \varphi(y)$ is the unique bounded steady state of \eqref{eq:main} satisfying $\varphi|_B \geq \theta + \delta$ for some ball $B \subset \hp$ of radius $R_{\trm{steady}}(\delta)$.
  \end{enumerate}
\end{theorem}
\noindent
We study the uniqueness of steady states in greater depth and in other domains in a forthcoming work.

Next, we consider propagation in \eqref{eq:main}.
In the whole space, solutions with sufficiently large initial data converge to $1$ locally uniformly as $t \to \infty$.
Moreover, the transition $u \to 1$ propagates asymptotically linearly in time at a speed $c_* > 0$ depending only on $f$.
We show that the dynamics of the transition $u \to \varphi$ in $\hp$ closely resemble this behavior.
As in the whole space, the asymptotic speed of propagation is $c_*$.
\begin{theorem}
  \label{thm:ASP}
  Throughout, let $u$ solve \eqref{eq:main} with ${0 \lneqq u_0 \leq 1}$ compactly supported.
  \begin{enumerate}[label = \textup{(\Alph*)}, leftmargin = 3em, labelsep = 0.6em, itemsep= 1ex, topsep = 1ex]
  \item
    \label{item:ASP-mono}
    Let $f$ be monostable with $\varrho \in [0, \infty)$.
    Then
    \begin{equation}
      \label{eq:lower-ASP}
      \limsup_{t \to \infty} \left[\sup_{\abs{(\tbf{x}', y)} \leq ct} \big|u(t, \tbf{x}', y) - \varphi(y)\big|\right] = 0 \quad \textrm{for all } c \in [0, c_*)
    \end{equation}
    and
    \begin{equation}
      \label{eq:upper-ASP}
      \limsup_{t \to \infty} \left[\sup_{\abs{(\tbf{x}', y)} \geq ct} u(t, \tbf{x}', y)\right] = 0 \quad \textrm{for all } c > c_*.
    \end{equation}

  \item
    \label{item:ASP-ign-bi}
    Let $f$ be ignition with $\varrho \in [0, \infty)$ or bistable with $\varrho = 0$.
    If $u_0 \leq \theta$, then $u(t, \anon) \to 0$ uniformly in $\bar \hp$ as $t \to \infty$.
    On the other hand, suppose that ${u_0|_B \geq \theta + \delta}$ for some $\delta \in (0, 1 - \theta)$ and some ball $B \subset \hp$ of radius $R_{\trm{steady}}(\delta)$.
    Then $u$ satisfies \eqref{eq:lower-ASP} and \eqref{eq:upper-ASP}.
  \end{enumerate}
\end{theorem}

Lou and Lu recently established the asymptotic speed of propagation in general convex cones for certain Fisher--KPP reactions with Dirichlet boundary conditions \cite{LL}.
They thus handle domains which are significantly more general than the half-space.
However, their results seem confined to so-called ``strong-KPP'' reactions with Dirichlet conditions.
The question of propagation in cones with Robin conditions and more general reactions remains open.
More broadly, the nature of invasion and propagation in general domains is an important open problem.

We now turn to traveling waves.
We say $\Phi \colon \bar \hp \to [0, 1]$ is a traveling wave of speed $c > 0$ and direction $\tbf e \in S^d$ if it is a function of $\tbf e \cdot \tbf{x}'$ and $y$ alone, $\Phi(\tbf e \cdot \tbf{x}' - ct, y)$ solves \eqref{eq:main}, and
\begin{equation}
  \label{eq:connection}
  \Phi(-\infty, y) = \varphi(y) \And \Phi(+\infty, y) = 0
\end{equation}
locally uniformly in $y \in [0, \infty)$.
That is, the wave moves parallel to $\partial \hp$ at speed $c$ in direction $\tbf e$, and connects the steady states $\varphi$ and $0$.
Its level sets are affine subspaces of codimension $2$, as $\Phi$ only depends on two spatial coordinates.
We may thus restrict our study of traveling waves to the half-\emph{plane}.
Then $d = 1$ and we denote position by $\tbf{x} = (x, y)$.

In one dimension, monostable reactions admit traveling waves precisely when $c \geq c_*$, where $c_*$ agrees with the asymptotic speed of propagation \cite{AW}.
In contrast, ignition and bistable reactions admit one-dimensional traveling waves precisely at speed $c_*$ \cite{AW, FM, Kanel1}.
We show nearly the same behavior in the absorbing half-plane.
\begin{theorem}
  \label{thm:TW}
  Let $d = 1$.
  No traveling wave has speed $c \in [0, c_*)$.
  Furthermore:
  \begin{enumerate}[label = \textup{(\Alph*)}, leftmargin = 3em, labelsep = 0.6em, itemsep= 1ex, topsep = 1ex]
  \item
    \label{item:TW-mono}
    Let $f$ be monostable with $\varrho \in [0, \infty).$
    Then there exists a traveling wave $\Phi$ of speed $c$ for each $c \geq c_*$.

  \item
    \label{item:TW-ign-bi}
    Let $f$ be ignition with $\varrho \in [0, \infty)$ or bistable with $\varrho = 0$.
    Then there exists a traveling wave $\Phi$ of speed $c_*$.
  \end{enumerate}
  In each case, $\Phi$ satisfies ${0 < \Phi < 1}$, $\partial_x \Phi < 0$, and $\partial_y \Phi > 0$ in $\hp$.
\end{theorem}
Significantly, we are unable to rule out ignition or bistable waves whose speeds exceed $c_*$.
In fact, we expect that so-called \emph{conical} waves of higher speed do exist.
This has been confirmed in the whole space; see, for instance, works of Hamel, Monneau, and Roquejoffre \cite{HMR1, HMR2} and Wang and Bu \cite{BW}.

As is clear from the theorem statements above, our results and methods vary between the monostable and ignition/bistable cases.
Although our monostable results are easier to state, their proofs rely on the ignition theory.
We therefore prove part (B) of each of our main theorems first.
We study ignition and bistable steady states and prove Theorem~\ref{thm:steady}\ref{item:steady-ign-bi} in Section~\ref{sec:steady-ign-bi}.
In Section~\ref{sec:strip-TW}, we develop the theory of traveling waves in \emph{strips} of bounded width.
Using waves in strips, we prove Theorems~\ref{thm:ASP}\ref{item:ASP-ign-bi} and \ref{thm:TW}\ref{item:TW-ign-bi} in Sections~\ref{sec:ASP-ign-bi} and \ref{sec:TW-ign-bi}, respectively.

We then pivot to monostable reactions.
We prove Theorem~\ref{thm:steady}\ref{item:steady-mono} for monostable steady states in Section~\ref{sec:steady-mono}.
Using ignition waves in strips, we prove Theorem~\ref{thm:ASP}\ref{item:ASP-mono} in Section~\ref{sec:ASP-mono}.
We close with monostable traveling waves and establish Theorem~\ref{thm:TW}\ref{item:TW-mono} in Section~\ref{sec:TW-mono}.

\section*{Acknowledgements}

This work was initiated while HB was the Poincar\'{e} visiting professor in 2019 in the Department of Mathematics at Stanford University, whose support is gratefully acknowledged.
We also thank l'\'{E}cole des hautes \'{e}tudes en sciences sociales for its generous support and hospitality.
CG was additionally supported by the Fannie and John~Hertz Foundation and by NSF grant DGE-1656518.

\section{Ignition and bistable steady states}
\label{sec:steady-ign-bi}

To begin, we let $f$ be ignition or bistable and consider the steady states of \eqref{eq:main} in various domains.
The simplest case is the half-line, which reduces to the ODE \eqref{eq:steady-ODE}.
Since $\hp = \R^d \times \R_+$, we can transfer states in the half-line to the half-space.
We then show that the half-space has no other steady states which exceed $\theta$ on large balls.
In Section~\ref{sec:strip-TW}, we will construct traveling waves in strips $\R \times [0, L]$.
We must therefore understand steady states in bounded intervals $[0, L]$.
This matter is quite delicate, and takes up the majority of this section.

\subsection{Steady states in the half-line}

\begin{lemma}
  \label{lem:steady-ODE-ign-bi}
  Let $f$ be ignition with $\varrho \in [0, \infty)$ or bistable with $\varrho = 0$.
  Then the ODE \eqref{eq:steady-ODE} has a unique nonzero bounded solution $\varphi$.
  Furthermore, $\varphi$ satisfies $0 \leq \varphi < 1$, $\varphi' > 0$, and $\varphi(+\infty) = 1$.
\end{lemma}

\begin{proof}
  Suppose $\phi$ is a nonzero bounded solution of \eqref{eq:steady-ODE}.
  Since $f$ vanishes outside $[0, 1]$, $\phi$ becomes affine linear if it exits this interval.
  Then $\abs{\phi}$ would grow without bound, a contradiction.
  As a consequence of the boundary condition, ${\phi(0) \in [0, 1)}$.
  Therefore, $\phi([0, \infty)) \subset [0, 1)$.
  Define
  \begin{equation*}
    y_\theta \coloneqq \inf\big\{y \in [0, \infty) \mid \phi(y) \geq \theta\big\},
  \end{equation*}
  recalling that $\theta$ is the smallest number for which $f|_{(\theta, 1)} > 0$.

  Suppose $f$ is ignition.
  Then $\phi$ is affine linear on $[0, y_\theta)$.
  We claim that $\phi' > 0$ on $(y_\theta, \infty)$.
  Otherwise, $\phi$ attains a local maximum.
  By concavity, it will bend back down until it reaches the value $\theta$ with a negative slope.
  Thereafter, $\phi$ will affine linearly decrease to $-\infty$, contradicting boundedness.
  So indeed $\phi' > 0$ on $(y_\theta, \infty)$ and $\phi((y_\theta, \infty)) \subset (\theta, 1)$.
  It follows that $\phi$ monotonically increases towards a zero of $f$.
  This zero can only be $1$, so $0\leq \phi < 1$, $\phi'>0$, and $\phi(+\infty) = 1$.
  
  Next, suppose $f$ is bistable and $\varrho = 0$, so $\phi(0) = 0$.
  Again, $\phi$ increases on $[0, y_\theta]$.
  If it attains a local maximum in $(y_\theta, \infty)$, uniqueness will force it to later hit $0$.
  Again, it will affine linearly decrease without bound, a contradiction.
  So $\phi' > 0$ and $\phi((y_\theta, \infty)) \subset (\theta, 1)$.
  Arguing as in the ignition case, we obtain $0 \leq \phi < 1$, $\phi' > 0$, and $\phi(+\infty) = 1$.
  
  We next prove uniqueness.
  Multiplying \eqref{eq:steady-ODE} by $\phi'$ and integrating over $\R_+$, we find
  \begin{equation}
    \label{eq:bulk-ign-bi}
    0 = \int_0^\infty \left\{\frac{1}{2}\big[(\phi')^2\big]' + f(\phi) \phi'\right\} \dn y = -\frac{1}{2} \phi'(0)^2 + \int_{\phi(0)}^1 f(s) \d s.
  \end{equation}
  Now suppose $\varrho = 0$.
  Then we can rearrange \eqref{eq:bulk-ign-bi} to obtain
  \begin{equation*}
    \phi'(0)^2 = 2 \int_0^1 f(r) \d r.
  \end{equation*}
  Thus the initial condition $(0, \phi'(0))$ is determined, and $\phi$ is unique.

  Suppose instead that $f$ is ignition and $\varrho > 0$.
  Using the boundary condition, \eqref{eq:bulk-ign-bi} yields
  \begin{equation}
    \label{eq:boundary-ign-bi}
    \varrho^{-2} = \frac{2}{\phi(0)^2} \int_{\phi(0)}^1 f(s) \d s.
  \end{equation}
  For $s \in (0, 1]$, we define the function
  \begin{equation}
    \label{eq:auxiliary-ign-bi}
    \Lambda(s) \coloneqq \frac{2}{s^2} \int_s^1 f(r) \d r,
  \end{equation}
  so that \eqref{eq:boundary-ign-bi} reads $\varrho^{-2} = \Lambda(\phi(0))$.
  
  Now $\frac{2}{s^2}$ is strictly decreasing while $\int_s^1 f(r) \d r$ is decreasing and nonzero, so their product $\Lambda$ is strictly decreasing.
  Furthermore, $\Lambda(0^+) = +\infty$ and $\Lambda(1) = 0.$
  Thus there exists a unique $s_\varrho \in (0, 1)$ such that $\varrho^{-2} = \Lambda(s_\varrho)$.
  By \eqref{eq:boundary-ign-bi}, the values $\phi(0) = s_\varrho$ and $\phi'(0) = \varrho^{-1} s_\varrho$ are determined.
  So again $\phi$ is unique.
  
  Finally, in each case we have produced a candidate initial condition $(\phi(0), \phi'(0))$.
  This immediately yields a nonzero bounded solution, so we have existence.
  We denote this solution by $\varphi$.
\end{proof}
We note that the Dirichlet assumption is crucial when $f$ is bistable.
After all, \eqref{eq:steady-ODE} may also admit oscillatory solutions when $\varrho > 0$.
In fact, the problem runs deeper.
Even if we restrict to the set of \emph{monotone} solutions, $\phi$ need not be unique.
Indeed, by \eqref{eq:boundary-ign-bi}, these solution are in bijective correspondence with the solutions to $\varrho^{-2} = \Lambda(s)$.
When $f$ is bistable, $\Lambda$ need not be monotone decreasing, so multiple values of $s$ may satisfy $\varrho^{-2} = \Lambda(s)$.
These constitute multiple initial conditions for bounded nonzero \emph{monotone} solutions of \eqref{eq:steady-ODE}.
This stronger form of nonuniqueness is the principal reason we only study bistable reactions with Dirichlet boundary conditions.

\subsection{Steady states in the half-space}

The extra degrees of freedom in $\hp$ make the classification of steady states more complex.
For instance, steady states in $\hp$ which are monotone in $y$ converge to steady states in $\R^d$ as $y \to \infty$.
The classification of such solutions under additional assumptions is known as De Giorgi's problem, and we anticipate exotic solutions in dimensions $d \geq 8$ \cite{dPKW}.
Nonetheless, we \emph{can} classify steady states which exceed $\theta$ on large balls.

First, we introduce one piece of notation.
Define the threshold
\begin{equation*}
  \vartheta \coloneqq \sup\left\{s \in [0, 1] \, \Big| \, \int_0^s f(r) \d r \leq 0\right\}.
\end{equation*}
Then $\vartheta = \theta$ when $f$ is ignition, while \ref{hyp:bi-pos} implies that $\vartheta \in (\theta, 1)$ when $f$ is bistable.
It is straightforward to classify steady states which exceed $\vartheta$ on large balls.
\begin{proposition}
  \label{prop:steady-ign-bi-sub}
  Let $f$ be ignition with $\varrho \in [0, \infty)$ or bistable with $\varrho = 0.$
  Then for all $\delta \in (0, 1 - \vartheta)$, there exists $R_{\trm{sub}}(\delta) > 0$ such that $\varphi = \varphi(y)$ is the unique bounded steady state of \eqref{eq:main} satisfying $\varphi|_B \geq \vartheta + \delta$ for some ball $B \subset \hp$ of radius $R_{\trm{sub}}(\delta)$.
\end{proposition}

\begin{proof}
  First, let $\phi$ be a bounded steady state.
  We claim that $0 \leq \phi \leq \varphi$.
  To see this, define
  \begin{equation*}
    M_- \coloneqq \min \big\{\inf \phi, \, 0\big\} \And M_+ \coloneqq \max\big\{\sup \phi, \, 1\big\},
  \end{equation*}
  which are both finite.
  Then $f(M_\pm) = 0$, so $M_-$ and $M_+$ are respectively sub- and supersolutions to \eqref{eq:steady}.
  If we evolve both under \eqref{eq:main}, they converge to bounded solutions of \eqref{eq:steady-ODE} as $t \to \infty$.
  By Lemma~\ref{lem:steady-ODE-ign-bi}, they must converge to $0$ and $\varphi$, respectively.
  Thus by comparison, $0 \leq \phi \leq \varphi$.
  
  Now fix $\delta \in (0, 1 - \vartheta)$.
  To prove uniqueness, we construct a nonzero compactly supported subsolution.
  Let $\mr{\phi}$ solve
  \begin{equation}
    \label{eq:Neumann-ODE}
    \mr{\phi}'' + f(\mr{\phi}) = 0, \quad \mr{\phi}(0) = \vartheta + \delta, \quad \mr{\phi}'(0) = 0.
  \end{equation}
  As in the proof of Lemma~\ref{lem:steady-ODE-ign-bi}, we multiply by $\mr{\phi}'$ and integrate to obtain
  \begin{equation*}
    \mr{\phi}'(y)^2 = 2 \int_{\mr{\phi}(y)}^{\vartheta + \delta} f(s) \d s \For y \geq 0.
  \end{equation*}
  Crucially, $\delta > 0$ prevents the right side from vanishing.
  Since $f(\vartheta) > 0$, $\mr{\phi}$ initially bends down.
  Then, $\mr{\phi}'$ is uniformly negative away from $y = 0$.
  Therefore $\mr{\phi}$ inevitably hits zero at some position $\mr K > 0$.
  Solving \eqref{eq:Neumann-ODE} on $\R_-$ as well, we obtain an even function which is positive precisely on $(-\mr K, \mr K)$.
  Hence $\mr{\phi}_+$ is a compactly supported subsolution to the ODE $\phi'' + f(\phi) = 0$.

  We now adapt this construction to higher dimensions.
  By the stability of ODEs, there exists $R_0 \geq \mr K$ such that the solution to
  \begin{equation*}
    \ti \phi'' + \frac{d}{R_0 + y} \ti \phi' +  f(\ti \phi) = 0, \quad \ti \phi(0) = \vartheta + \delta, \quad \ti \phi'(0) = 0
  \end{equation*}
  also hits $0$ at some position $K_0 > 0$.
  Furthermore, $\ti \phi' < 0$ on $(0, K_0)$.

  Using $\ti \phi$, we construct a radial subsolution $v$.
  Let $r \coloneqq \abs{\tbf{x}}$ denote the radial coordinate and define
  \begin{equation}
    \label{eq:radial-subsolution}
    v(r) \coloneqq
    \begin{cases}
      \vartheta + \delta & \textrm{if } r \leq R_0,\\
      \ti\phi\big(r - R_0\big) & \textrm{if } r \in (R_0, R_0 + K_0),\\
      0 & \textrm{if } r \geq R_0 + K_0.
    \end{cases}
  \end{equation}
  Let $R_{\trm{sub}}(\delta) \coloneqq R_0 + K_0$.
  Writing the Laplacian in polar coordinates, we see that $v$ is a nonzero subsolution of \eqref{eq:main} supported in a ball of radius $R_{\trm{sub}}$ and $v \leq \vartheta + \delta$.

  Now suppose that $\phi$ is a bounded steady state of \eqref{eq:main} such that $\phi|_B \geq \vartheta + \delta$ for some ball $B \subset \hp$ of radius $R_{\trm{sub}}$.
  Since $\phi \geq 0$, we have $\phi \geq (\vartheta + \delta) \tbf{1}_B$.
  
  Now let $\tbf{x}_0 = (\tbf{x}_0', y_0)$ denote the center of $B$, and let
  \begin{equation*}
    \tau_{\tbf{x}_0} v \coloneqq v(\anon - \tbf{x}_0)
  \end{equation*}
  denote the translation of $v$ by $\tbf{x}_0$.
  Then
  \begin{equation*}
    \phi \geq (\vartheta + \delta) \tbf{1}_B \geq \tau_{\tbf{x}_0} v.
  \end{equation*}
  We now use the sliding method of \cite{BN}.
  If we continuously vary $\tbf{x}_0'$, the strong maximum principle implies that the subsolution $\tau_{\tbf x_0} v$ can never touch the solution $\phi$ from below.
  Since $R_0 \geq \mr K$, it follows that
  \begin{equation*}
    \phi(\tbf{x}', y) \geq (\vartheta + \delta) \tbf{1}_{[-\mr K, \mr K]}(y - y_0) \geq \tau_{y_0} \mr \phi_+(y) \ForAll (\tbf{x}', y) \in \hp.
  \end{equation*}

  By construction, $\tau_{y_0} \mr \phi_+$ is a subsolution of the one-dimensional problem.
  If we solve the one-dimensional parabolic problem with initial data $\tau_{y_0} \mr \phi_+$, the solution is thus increasing in time.
  On the other hand, $\tau_{y_0} \mr \phi_+$ lies beneath the bounded supersolution $1$, so its long-time limit is also bounded.
  Hence, the parabolic solution converges to a bounded nonzero steady state in $\R_+$ as $t \to \infty$.
  By Lemma~\ref{lem:steady-ODE-ign-bi}, this steady state is $\varphi$.
  Applying the comparison principle, we obtain $\phi \geq \varphi$.
  We showed above that $\phi \leq \varphi$, so in fact $\phi = \varphi$.
\end{proof}

If $f$ is ignition, $\vartheta = \theta$, and the conclusion of Theorem~\textup{\ref{thm:steady}\ref{item:steady-ign-bi}} follows from Proposition~\ref{prop:steady-ign-bi-sub}.
When $f$ if bistable, however, we must work harder to lower the uniqueness threshold to $\theta$.
To do so, we study extinction and invasion in \eqref{eq:main}.
\begin{lemma}
  \label{lem:invasion-ign-bi}
  Let $f$ be ignition with $\varrho \in [0, \infty)$ or bistable with $\varrho = 0$.
  If $0 \leq u_0 \leq \theta$ is compactly supported, then $u(t, \anon) \to 0$ uniformly as $t \to \infty$.
  On the other hand, for any $\delta \in (0, 1 - \theta)$, there exists $R_{\trm{steady}}(\delta) > 0$ such that the following holds.
  If $(\theta + \delta) \tbf{1}_B \leq u_0 \leq 1$ for some ball $B \subset \hp$ of radius $R_{\trm{steady}}(\delta)$, then
  \begin{equation}
    \label{eq:invasion-ign-bi}
    \lim_{t \to \infty} u(t, \anon) = \varphi
  \end{equation}
  locally uniformly in $\hp$.
\end{lemma}

\begin{proof}
  First suppose that $0 \leq u_0 \leq \theta$ is compactly supported.
  By the comparison principle, $u \leq \theta$ for all time.
  Since $f \leq 0$ on $[0, \theta]$, we can use the heat evolution in the whole space as a supersolution.
  But $u_0$ is compactly supported, so $\e^{t \Delta} u_0 \to 0$ uniformly as $t \to \infty$.
  Thus the same holds for $u(t, \anon)$.

  Next, suppose $u_0 \leq 1$.
  Let $u^1$ denote the solution of \eqref{eq:main} with initial data $1$.
  Since $1$ is a supersolution of \eqref{eq:main}, $u^1$ is decreasing in time towards a bounded solution of \eqref{eq:steady-ODE}.
  By Lemma~\ref{lem:steady-ODE-ign-bi} and the comparison principle, this solution is $\varphi$.
  Since $u \leq u^1$ by comparison, it follows that
  \begin{equation}
    \label{eq:invasion-ign-bi-upper}
    \limsup_{t \to \infty} u(t, \anon) \leq \lim_{t \to \infty} u^1(t, \anon) = \varphi.
  \end{equation}

  Now consider $\delta \in (0, 1 - \vartheta)$.
  We use the compactly supported radial subsolution $v$ defined in \eqref{eq:radial-subsolution}.
  It satisfies $v \leq \vartheta + \delta$ and is supported on a ball of radius ${R_{\trm{sub}}(\delta) > 0}$.
  If $u^v$ denotes the solution to \eqref{eq:main} with initial data $v$, then $u^v$ increases in time towards a bounded steady state of \eqref{eq:main} which exceeds $v$.
  As shown in the proof of Proposition~\ref{prop:steady-ign-bi-sub}, the only such state is $\varphi$.
  So
  \begin{equation}
    \label{eq:invasion-ign-bi-lower}
    \liminf_{t \to \infty} u(t, \anon) \geq \lim_{t \to \infty} u^v(t, \anon) = \varphi.
  \end{equation}
  Combining this with \eqref{eq:invasion-ign-bi-upper}, we obtain \eqref{eq:invasion-ign-bi}.
  Moreover, the convergence is locally uniform by Dini's theorem.
  Since $\theta = \vartheta$ when $f$ is ignition, this concludes the ignition case with $R_{\trm{steady}}(\delta) \coloneqq R_{\trm{sub}}(\delta)$.

  Now suppose $f$ is bistable and fix $\delta \in (0, 1 - \theta)$.
  We wish to show that initial data above $\theta + \delta$ on large balls \emph{also} survive and eventually converge to $\varphi$.
  This is more subtle, because we may not be able to fit $v$ under $u_0$.
  Let $S(t)$ solve $\dot S = f(S)$ with $S(0) = \theta + \delta$.
  Then $S(t)$ agrees with the whole space evolution from the constant $\theta + \delta$ and converges to $1$ as $t \to \infty$.
  In particular, if we fix $\delta' \in (0, 1 - \vartheta)$, $S$ will exceed $\vartheta + \delta'$ at some fixed time $T > 0$.
  Now suppose that $u_0 = \theta + \delta$ on an enormous ball $B$.
  Then $u$ will resemble the free space evolution near the center of $B$, at least on a bounded time interval.
  This is enough to push $u$ above $(\vartheta + \delta') \tbf{1}_{B'}$ for some ball $B'$ of radius $R_{\trm{sub}}(\delta')$, allowing us to apply the previous argument.

  To make this reasoning rigorous, let $B(k)$ denote the ball of radius $k$ centered at $(\tbf{0}, k) \in \R^{d + 1}$, so that $B(k) \subset \hp$.
  Let $u_k$ solve \eqref{eq:main} with initial data $(\theta + \delta) \tbf{1}_{B(k)}$, and let
  \begin{equation*}
    \ti{u}_k(\tbf x', y) \coloneqq u_k(\tbf x', y + k)
  \end{equation*}
  denote its shift to the origin.
  Let $B'(k)$ denote the shifted ball of radius $k$, which is centered at the origin.
  In this shifted frame, the boundary of $\hp$ is moving away to infinity as $k \to \infty$, and $\ti u_k(0, \tbf x) \to \theta + \delta$.
  Thus on the compact set ${\bar{B}'(R_{\trm{sub}}(\delta')) \times [0, T]}$, $\ti u_k$ converges to $S$ uniformly as $k \to \infty$.
  Since ${S(T) > \vartheta + \delta'}$, there exists $K \in \N$ such that
  \begin{equation*}
    \ti u_K(T, \tbf{x}) \geq (\vartheta + \delta') \tbf{1}_{B'(R_{\trm{sub}}(\delta'))}(\tbf{x}).
  \end{equation*}
  By our preceding argument, $u_K$ satisfies \eqref{eq:invasion-ign-bi-lower}.

  Finally, suppose $u_0 \geq (\theta + \delta) \tbf{1}_{B}$ for some other ball $B \subset \hp$ of radius $K$.
  Then we can move the boundary of $\hp$ to touch $B$ and solve to obtain a subsolution.
  This subsolution will be a shift of $u_K$, so $u$ exceeds a shift of $u_K$.
  Since $u_K$ satisfies \eqref{eq:invasion-ign-bi-lower}, $u$ will eventually be close to $1$ on a large ball.
  Then we can argue as in the ignition case, and $u$ also satisfies \eqref{eq:invasion-ign-bi-lower}.
  In combination with \eqref{eq:invasion-ign-bi-upper}, we obtain \eqref{eq:invasion-ign-bi} with $R_{\trm{steady}}(\delta) \coloneqq K$.
\end{proof}

We can now complete our analysis of ignition and bistable steady states on $\hp$.
\begin{proof}[Proof of Theorem~\textup{\ref{thm:steady}\ref{item:steady-ign-bi}}]
  If $f$ is ignition, $\vartheta = \theta$, so the conclusion of Theorem~\ref{thm:steady}\ref{item:steady-ign-bi} follows from Proposition~\ref{prop:steady-ign-bi-sub} with $R_{\trm{steady}} \coloneqq R_{\trm{sub}}$.
  
  Therefore, suppose $f$ is bistable and fix $\delta \in (0, 1 - \theta)$.
  Let $\phi$ be a bounded steady state of \eqref{eq:main} such that $\phi|_B \geq \theta + \delta$ on a ball $B \subset \hp$ of the radius $R_{\trm{steady}}(\delta)$ from Lemma~\ref{lem:invasion-ign-bi}.
  As shown in the proof of Proposition~\ref{prop:steady-ign-bi-sub}, $0 \leq \phi \leq \varphi$.
  So $\phi \geq (\theta + \delta) \tbf{1}_{B}$.
  Let $u$ solve \eqref{eq:main} with $u_0 \coloneqq (\theta + \delta) \tbf{1}_{B}$.
  Then by Lemma~\ref{lem:invasion-ign-bi}, $u(t, \anon) \to \varphi$ locally uniformly as $t \to \infty$.
  On the other hand, the comparison principle implies that $\phi \geq u(t, \anon)$ for all $t \geq 0$.
  Therefore $\varphi \leq \phi \leq \varphi$, as desired.
\end{proof}

\subsection{Steady states in bounded intervals}

To constrain the asymptotic speed of propagation and construct traveling waves, we will consistently use waves in the strips $\R \times [0, L]$ as subsolutions.
As $x \to \pm \infty,$ these waves converge to solutions of:
\begin{equation}
  \label{eq:strip-ODE}
  \phi_L'' + f(\phi_L) = 0, \quad \phi_L'(0) = \varrho^{-1} \phi_L(0), \quad \phi_L(L) = 0.
\end{equation}
For general $L$, this equation may have no solutions or many.
However, the situation simplifies when $L$ is large.
Then, \eqref{eq:strip-ODE} admits precisely two nonzero steady states.
\begin{lemma}
  \label{lem:strip-ODE-ign-bi}
  Let $f$ be ignition with $\varrho \in [0, \infty)$ or bistable with $\varrho = 0$.
  Then there exists $A_{\trm{ODE}} > 0$ such that for all $L > A_{\trm{ODE}}$, \eqref{eq:strip-ODE} admits precisely two nonzero solutions $\varphi_L$ and $\psi_L$.
  These are strictly ordered: $0 < \psi_L < \varphi_L < \varphi$ in $(0, L)$.
  Moreover,
  \begin{equation}
    \label{eq:strip-steady-convergence}
    \limsup_{L \to \infty} \, \sup_{y \in [0, \, L/2]} \abs{\varphi_L(y) - \varphi(y)} = 0
  \end{equation}
  and
  \begin{equation}
    \label{eq:strip-psi-vartheta}
    \lim_{L \to \infty} \sup_{y \in [0, L]} \psi_L(y) = \vartheta.
  \end{equation}
\end{lemma}

\begin{proof}
  Let $\phi$ be a nonzero solution to \eqref{eq:strip-ODE}.
  By arguments similar to those in the proof of Lemma~\ref{lem:steady-ODE-ign-bi}, $\phi((0, L)) \subset (0, 1)$.

  We use the shooting method.
  It is thus convenient to view $\phi$ as a function of its initial slope rather than of $L$.
  Let $\phi^\al$ denote the solution of the initial value problem
  \begin{equation}
    \label{eq:ODE-IVP}
    (\phi^\al)'' + f(\phi^\al) = 0, \quad \phi^\al(0) = \varrho \al, \And (\phi^\al)'(0) = \al
  \end{equation}
  for $\al \in \R$.
  We can discard $\al \leq 0$, for then $\phi^\al$ is affine linear and nonpositive.
  
  Suppose $\al > 0$.
  Then $(\phi^\al)'$ is initially positive.
  Let $K^\al$ denote the location of its first zero, if it exists.
  Otherwise let $K^\al = +\infty$.
  Multiplying \eqref{eq:ODE-IVP} by $(\phi^\al)'$ and integrating, we obtain
  \begin{equation}
    \label{eq:ODE-first-integral}
    \big[(\phi^\al)'\big]^2 - \al^2 = -2\int_{\varrho \al}^{\phi^\al} f(s) \d s \quad \textrm{on } [0, K^\al].
  \end{equation}
  Now let $\bar{\al} \coloneqq \varphi'(0)$.
  By \eqref{eq:bulk-ign-bi}, $\bar{\al}$ satisfies
  \begin{equation*}
    \bar{\al}^2 = 2\int_{\varrho \bar \al}^1 f(s) \d s.
  \end{equation*}
  
  We claim that $\phi^\al$ hits the value $1$ with positive slope, and thus exits the range $(0, 1)$, when $\al > \bar{\al}$.
  Such solutions will never satisfy the second boundary condition in \eqref{eq:strip-ODE}, so we can confine our attention to $\al \in (0, \bar{\al})$.
  To see this, first suppose $\varrho = 0$ and $\al > \bar \al$.
  Then \eqref{eq:ODE-first-integral} implies
  \begin{equation*}
    \big[(\phi^\al)'\big]^2 \geq \al^2 - 2\int_{0}^{1} f(s) \d s > \bar \al^2 - 2\int_{0}^{1} f(s) \d s = 0.
  \end{equation*}
  So indeed $\phi^\al$ still has positive slope when it attains the value $1$.
  
  Now suppose $\varrho > 0$, so that $f$ is ignition.
  We recall $\Lambda$ defined in \eqref{eq:auxiliary-ign-bi}.
  In the proof of Lemma~\ref{lem:steady-ODE-ign-bi}, we showed that $\Lambda$ is strictly decreasing and $\Lambda(\varrho \bar \al) = \varrho^{-2}$.
  Assuming $\al > \bar \al$, \eqref{eq:ODE-first-integral} yields
  \begin{equation*}
    \left[\frac{(\phi^\al)'}{\varrho \al}\right]^2 \geq \varrho^{-2} - \Lambda(\varrho \al) > \varrho^{-2} - \Lambda(\varrho \bar \al) = 0.
  \end{equation*}
  This proves the claim, and we may assume that $\al \in (0, \bar\al)$.
  
  For such $\al$, the solution $\phi^\al$ rises initially, attains its maximum at position ${K^\al \in (0, \infty)}$, and then falls back to $0$ at some position $L^\al \in (0, \infty)$.
  Let $s^\al$ denote its maximal value, so that
  \begin{equation*}
    \phi^\al(K^\al) = s^\al \And (\phi^\al)'(K^\al) = 0.
  \end{equation*}

  \noindent{\bf Step 1.}
  We will show that $K^\al$ increases monotonically with $\al$ when $\al$ is sufficiently close to $\bar{\al}$.
  We define
  \begin{equation*}
    F^\al(s) \coloneqq \int_{\varrho \al}^s f(r) \d r,
  \end{equation*}
  so that \eqref{eq:ODE-first-integral} reads
  \begin{equation*}
    \big[(\phi^\al)'\big]^2 = \al^2 - 2F^\al(\phi^\al) \quad \textrm{on } [0, K^\al].
  \end{equation*}
  Separating variables, we obtain an implicit equation for $\phi^\al$:
  \begin{equation}
    \label{eq:implicit}
    y = \int_{\varrho \al}^{\phi^\al(y)} \frac{\dn s}{\sqrt{\al^2 - 2F^\al(s)}} \quad \textrm{for } y \in [0, K^\al].
  \end{equation}
  In particular,
  \begin{equation}
    \label{eq:length}
    K^\al = \int_{\varrho \al}^{s^\al} \frac{\dn s}{\sqrt{\al^2 - 2F^\al(s)}} \, .
  \end{equation}
  In the following, we use the notation $\dot{G} \coloneqq \partial_\al G$ and $G' \coloneqq \partial_s G$ for any function $G$ depending on $\al$ or $s$.
  
  We first consider the dependence of the maximal value $s^\al$ on $\al$.
  Evaluating \eqref{eq:ODE-first-integral} at $y = K^\al$, we find
  \begin{equation}
    \label{eq:strip-boundary}
    \al^2 = 2 F^\al(s^\al).
  \end{equation}
  Differentiating with respect to $\al$,
  \begin{equation}
    \label{eq:max-deriv}
    \dot{s}^\al = \frac{\al + \varrho f(\varrho \al)}{f(s^\al)} > 0.
  \end{equation}
  Significantly, this derivative tends to $\infty$ as $\al \nearrow \bar \al$, for then $f(s^\al) \to  f(1) = 0$.

  We now study \eqref{eq:length} in detail.
  Notice that \eqref{eq:strip-boundary} implies that the integral in \eqref{eq:length} is improper at its right endpoint $s^\al$.
  Furthermore, $(F^\al)'(1) = f(1) = 0$, so $\al^2 - 2F^\al(s)$ has a double root at $s = s^\al = 1$ when $\al = \bar{\al}$.
  Thus the integrand of \eqref{eq:length} becomes logarithmically nonintegrable at $\al = \bar{\al}$.
  That is,  $K_{\bar{\al}} = \infty$.
  Of course, this agrees with the definition of $\bar{\al}$: it is the initial slope of $\varphi$, which approaches its supremum at $y = + \infty$.
  Since $K^\al \to \infty$ as $\al \nearrow \bar{\al}$, it is natural to expect that $K^\al$ increases \emph{monotonically} in $\al$ when $\al$ is sufficiently close to $\bar{\al}$.
  We verify this through direct calculation.

  The ignition and bistable cases require slightly different treatments, so suppose for the moment that $f$ is ignition.
  The principal difficulty in the analysis of \eqref{eq:length} is the singularity of the integrand at the moving endpoint $s^\al$.
  To fix this, we change variables:
  \begin{equation*}
    K^\al = \int_0^{s^\al - \varrho \al} \frac{\dn z}{\sqrt{\al^2 - 2F^\al(s^\al - z)}} \,.
  \end{equation*}
  Then, the definition of $F^\al$ implies
  \begin{equation}
    \label{eq:length-deriv}
    \dot{K}^\al = \frac{\dot{s}^\al - \varrho}{\al} + \int_0^{s^\al - \varrho \al} \frac{\dot{s}^\al f(s^\al - z) - \varrho f(\varrho \al) - \al}{[\al^2 - 2 F^\al(s^\al - z)]^{3/2}} \d z.
  \end{equation}
  By \eqref{eq:max-deriv},
  \begin{equation*}
    \frac{\dot{s}^\al - \varrho}{\al} \to \infty \quad \textrm{as } \al \nearrow \bar{\al}.
  \end{equation*}
  Therefore, $\dot{K}^\al$ will be positive for $\al$ near $\bar{\al}$ unless the second term in \eqref{eq:length-deriv} becomes negatively divergent.
  Divergence can come from two sources in the integral: $\dot{s}^\al$ and the singularity at $z = 0$.
  But $\dot{s}^\al$ is paired with $f(s^\al - z) \geq 0$, and thus can only contribute a positive divergence.
  So, it suffices to analyze the integrand in \eqref{eq:length-deriv} around $z = 0$.

  First, note that \eqref{eq:max-deriv} implies that the numerator in the integrand vanishes at $z = 0$.
  Differentiating in $z$, we find
  \begin{equation}
    \label{eq:numerator-z-deriv}
    \partial_z \left[\dot{s}^\al f(s^\al - z) - \varrho f(\varrho \al) - \al\right] = -\dot{s}^\al f'(s^\al - z).
  \end{equation}
  Now, \ref{hyp:ign} implies that $f'$ is negative near $1$.
  By \eqref{eq:numerator-z-deriv}, the integrand in \eqref{eq:length-deriv} is \emph{positive} when $\al \approx \bar{\al}$ and $z \approx 0$.
  Thus, the integrand in \eqref{eq:length-deriv} is bounded from below when $\al \approx \bar{\al}$.
  Since the first term in \eqref{eq:length-deriv} diverges to $+\infty$, there exists $\al_1 \in (0, \bar{\al})$ such that
  \begin{equation*}
    \dot{K}^\al > 0 \quad \textrm{for all } \al \in [\al_1, \bar \al).
  \end{equation*}

  Now suppose $f$ is bistable and $\varrho = 0$.
  We write \eqref{eq:length} as
  \begin{equation}
    \label{eq:bi-decomp}
    K^\al = \int_0^{\theta} \frac{\dn s}{\sqrt{\al^2 - 2F(s)}} + \int_0^{s^\al - \theta} \frac{\dn z}{\sqrt{\al^2 - 2F(s^\al - z)}} \,.
  \end{equation}
  Note that we are free to write $F$ rather than $F^\al$, because the antiderivative is independent of $\al$ when $\varrho = 0$.
  The first term in \eqref{eq:bi-decomp} varies smoothly in $\al$ away from $\al = 0$, and we can treat the second term as in the ignition case because $f|_{(\theta, 1)} > 0$.
  \\

  \noindent{\bf Step 2.}
  Next, we show that the positive root $L^\al$ of $\phi^\al$ is also monotone in $\al$ when $\al \approx \bar\al$.
  ODE uniqueness implies that $L^\al = 2K^\al$ when $\varrho = 0$, so we need only consider the Robin case $\varrho > 0$.
  
  Define $\beta \coloneqq -(\phi^\al)'(L^\al)$.
  Then $\phi_0^\beta(x) \coloneqq \phi^\al(L^\al - x)$ solves \eqref{eq:ODE-IVP} with $\beta$ in the place of $\al$ and $\varrho = 0$.
  Thus $\phi_{L^\al}$ is a Robin solution $\phi^\al$ or a Dirichlet solution $\phi_0^\beta$, depending on our point of view.

  Let $\varphi_0$ denote the unique nonzero bounded solution to \eqref{eq:steady-ODE} with $\varrho = 0$, and let $K_0^\beta$ and $s_0^\beta$ denote the Dirichlet analogues of $K^\al$ and $s^\al$.
  Then, by Step $1$, there exists an interval of initial slopes $\beta$ close to $\bar \beta \coloneqq \varphi_0'(0)$ for which $K_0^\beta$ and $s_0^\beta$ vary monotonically in $\beta$.
  Furthermore, we can turn things on their head and instead view our parameters as functions of the maximal value $s$, which, by \eqref{eq:max-deriv}, varies monotonically with $\beta$.
  Thus, there is an interval $[s_+, 1)$ on which $K_0$ and $\beta$ vary monotonically with $s$.

  In the original problem, we vary the slope $\al$ at the Robin boundary $y = 0$ within the interval $[\al_1, \bar \al)$.
  Then $K^\al \gg 1$ and $s^\al \approx 1$ vary monotonically with $\al$.
  Increasing $\al_1$ if need be, we can assume that $s^\al > s_+$.
  Then $s^\al$ monotonically determines $\beta$ and the Dirichlet length $K_0^\beta$.
  It follows that $L^\al = K^\al + K_0^\beta$ is the sum of two functions which are increasing in $\al$.
  We have thus shown that
  \begin{equation*}
    \dot{L}^\al > 0 \quad \textrm{for } \al \in [\al_1, \bar \al).
  \end{equation*}

  \noindent{\bf Step 3.}
  We must now study $\al \in (0, \al_1]$.
  If $\al$ is bounded away from the endpoints $0$ and $\bar{\al},$ \eqref{eq:length} shows that $L^\al$ is uniformly bounded.
  It thus suffices to understand the regime $\al \approx 0$.
  We show that $L^\al$ also diverges to $\infty$ when $\al \searrow 0$, and that it does so monotonically once $\al$ is sufficiently small.
  We handle the ignition and bistable cases separately.
  
  Suppose $f$ is ignition.
  By \eqref{eq:strip-boundary}, $s^\al \to \vartheta = \theta$ as $\al \to 0$.
  When $\al < \frac{\theta}{\varrho}$, our initial value $\varrho \al$ lies below the ignition temperature $\theta$.
  Since $f$ vanishes below $\theta$, $\phi^\al$ is affine linear on the initial interval $[0, \theta/\al - \varrho]$.
  After, it enters the range $(\theta, 1)$, where $f$ is positive.
  In fact, by \ref{hyp:ign} and \ref{hyp:ign-derivs}, $f$ is \emph{monostable} when viewed on the restricted interval $[\theta, 1]$.
  Let $\ti f(s) \coloneqq f(s + \theta)$, and analogously define $\ti{K}^\al$, $\ti{s}^\al$, etc.
  Then
  \begin{equation*}
    L^\al = \left(\frac{\theta}{\al} - \varrho + \ti{K}^\al\right) + \left(\ti{K}^\al + \frac{\theta}{\al}\right) = \frac{2\theta}{\al} + 2\ti K^\al - \varrho,
  \end{equation*}
  so
  \begin{equation}
    \label{eq:ign-length-psi}
    \dot{L}^\al = - \frac{2\theta}{\al^2} + 2 \dot{\ti K}^\al.
  \end{equation}
  We will show that
  \begin{equation}
    \label{eq:mono-length-small}
    \dot{\ti K}^\al = \m O\big(\al^{-1}\big).
  \end{equation}
  By \eqref{eq:ign-length-psi}, this implies that $L^\al \to \infty$ as $\al \to 0$ and $\dot{L}^\al < 0$ when $\al \in (0, \al_0]$ for some $\al_0 \in (0, \al_1)$.
  
  Since our solution $\phi^\al$ enters the range $(\theta, 1)$ from below, we are effectively considering Dirichlet solutions for the monostable reaction $\ti{f}$.
  By \ref{hyp:ign-derivs}, $\ti f'(0^+) > 0$.
  So $\ti f(s) \sim \ti f'(0) s$ and $\ti F(s) \sim \frac 1 2 \ti f'(0) s^2$ as $s \searrow 0$.
  By \eqref{eq:strip-boundary} and \eqref{eq:max-deriv}, we have
  \begin{equation*}
    \ti s^\al \sim \ti f'(0)^{-\frac 1 2} \al \And \dot{\ti s}^\al \sim \ti f'(0)^{-\frac 1 2}.
  \end{equation*}
  Hence, the first term in \eqref{eq:length-deriv} is $\m O(\al^{-1})$.
  For the second term in \eqref{eq:length-deriv}, we recall that the numerator of the integrand vanishes at $z = 0$ and has derivative
  \begin{equation*}
    -\dot{\ti s}^\al \ti f'(s^\al - z) \sim - \sqrt{\ti f'(0)}.
  \end{equation*}
  So,
  \begin{equation*}
    \big|\dot{\ti s}^\al \ti f(\ti s^\al - z) - \al\big| \leq C z
  \end{equation*}
  for some $C \geq 1$ that may change from expression to expression, but is independent of $\al$ and $z$.
  In the denominator, we have
  \begin{equation*}
    \al^2 - 2\ti F(\ti s^\al - z) \sim \al^2 - \ti f'(0) \big(\ti f'(0)^{-\frac{1}{2}} \al - z\big)^2 \sim 2 \sqrt{\ti f'(0)} \al z - z^2 \geq C^{-1} \al z.
  \end{equation*}
  It follows that
  \begin{equation*}
    \big|\dot{\ti K}^\al\big| \leq C\al^{-1} + C\int_0^{C \al} \frac{z}{\al^{3/2} z^{3/2}} \d z \leq C \al^{-1}.
  \end{equation*}
  This confirms \eqref{eq:mono-length-small}.

  Now consider a bistable reaction with $\varrho = 0$.
  Again, we show that $L^\al \to +\infty$ as $\al \searrow 0$, monotonically when $\al$ is small.
  Since we have Dirichlet conditions, $L^\al = 2K^\al$.
  By \eqref{eq:strip-boundary}, $s^\al \searrow \vartheta$ as $\al \searrow 0$.
  Hence we can rewrite \eqref{eq:length} as
  \begin{equation}
    \label{eq:bi-length}
    K^\al = \int_0^{\vartheta} \frac{\dn s}{\sqrt{\al^2 - 2F(s)}} + \int_\vartheta^{s^\al} \frac{\dn s}{\sqrt{\al^2 - 2F(s)}}\,.
  \end{equation}
  Now, \ref{hyp:bi} and \ref{hyp:bi-derivs} imply
  \begin{align}
    \nonumber
    F(s) &\sim \frac 1 2 f'(0) s^2 = -\frac 1 2 \abs{f'(0)}s^2 \quad \textrm{as } s \searrow 0,\\
    \label{eq:F-bi}
    F(s) &\sim f(\vartheta) (s - \vartheta)  \quad \textrm{as } s \to \vartheta.
  \end{align}
  Since $F$ vanishes to second order at $s = 0$, the first term in \eqref{eq:bi-length} diverges as $\al \searrow 0$.
  So indeed $K^\al \to \infty$.

  We now show monotonicity.
  Differentiating the first term in \eqref{eq:bi-length}, we can compute
  \begin{equation}
    \label{eq:bi-length-main}
    \der{}{\al} \int_0^{\vartheta} \frac{\dn s}{\sqrt{\al^2 - 2F(s)}} = -\al \int_0^\vartheta \big[\al^2 - 2 F(s)\big]^{-\frac 3 2} \d s \sim -C_0 \al^{-1}
  \end{equation}
  for some constant $C_0 > 0$ depending on $f$.
  On the other hand,
  \begin{align*}
    \der{}{\al} \int_\vartheta^{s^\al} \frac{\dn s}{\sqrt{\al^2 - 2F(s)}} &= \der{}{\al} \int_0^{s^\al - \vartheta} \frac{\dn z}{\sqrt{\al^2 - 2F(s^\al - z)}}\\
                                                                          &= \frac{\dot{s}^\al}{\al} + \int_0^{s^\al - \vartheta} \frac{\dot{s}^\al f(s^\al - z) - \al}{[\al^2 - 2F(s^\al - z)]^{3/2}} \d z.
  \end{align*}
  Again, the numerator of the integrand is bounded by $C z$, and the denominator involves
  \begin{equation*}
    \al^2 - 2F(s^\al - z) \sim 2 f(\vartheta) z.
  \end{equation*}
  Also, \eqref{eq:strip-boundary}, \eqref{eq:max-deriv}, and \eqref{eq:F-bi} imply that
  \begin{equation*}
    s^\al - \vartheta \sim \frac{\al^2}{2f(\vartheta)} \And \dot{s}^\al \sim \frac{\al}{f(\vartheta)}.
  \end{equation*}
  Therefore
  \begin{equation*}
    \abs{\der{}{\al} \int_\vartheta^{s^\al} \frac{\dn s}{\sqrt{\al^2 - 2F(s)}}} \leq C + C \al \int_0^{C \al^2} \frac{\dn z}{\sqrt{z}} \leq C.
  \end{equation*}
  By \eqref{eq:bi-length-main}, we have $\dot{K}^\al < 0$ provided $\al$ is sufficiently small, as desired.

  In each case, we have shown that there exist $0 < \al_0 < \al_1 < \bar{\al}$ such that $L^\al$ is monotone on $(0, \al_0)$ and $(\al_1, \bar{\al})$.
  Moreover, $L^\al$ diverges to $+\infty$ at $0$ and $\bar{\al}$, while $L^\al$ is bounded on $[\al_0, \al_1]$.
  Thus if we define
  \begin{equation*}
    A_{\trm{ODE}} \coloneqq \sup_{\al \in [\al_0, \al_1]} L^\al,
  \end{equation*}
  then the ODE \eqref{eq:strip-ODE} admits precisely two nonzero solutions, which we call $\varphi_L$ and $\psi_L$, when $L > A_{\trm{ODE}}$.
  Their initial slopes are close to $\bar{\al}$ and $0$, respectively.
  Adjusting $\al_0$ and $\al_1$ if need be, we can thus arrange that $\varphi_L > \psi_L$ in $(0, L)$.
  \\

  \noindent{\bf Step 4.}
  We now turn to \eqref{eq:strip-steady-convergence}.
  To prove the existence of a limit, we show that $\varphi_L$ increases in $L$ when $L > A_{\trm{ODE}}$.
  That is:
  \begin{equation}
    \label{eq:increasing}
    \varphi_{L'} > \varphi_L \quad \textrm{on } (0, L)
  \end{equation}
  when $L' > L > A_{\trm{ODE}}$.
  Let $\al' > \al > \al_1$ denote the corresponding initial slopes, and let $K' > K$ and $s' > s$ denote the maximal positions and values for $\varphi_{L'}$ and $\varphi_L$, respectively.

  Differentiating \eqref{eq:implicit} with respect to $\al$, we find
  \begin{equation*}
    \frac{\dot{\phi}^\al(y)}{\sqrt{\al^2 - 2F^\al(\phi^\al(y))}} = \frac{\varrho}{\al} + \int_{\varrho \al}^{\phi^\al(y)}\frac{\al + \varrho f(\varrho \al)}{[\al^2 - 2F^\al(r)]^{3/2}} \d r > 0.
  \end{equation*}
  We have used the fact that $\varrho = 0$ when $f$ is bistable, so the term $\varrho f(\varrho \al)$ is always nonnegative.
  Thus, $\phi^\al(y)$ is increasing in $\al$ for all $y \in (0, K^\al)$.
  It follows that $\varphi_{L'} > \varphi_L$ in $(0, K')$.
  If we view these profiles in reverse from their endpoints, the same reasoning (with $\varrho = 0$) implies
  \begin{equation}
    \label{eq:reverse}
    \varphi_{L'}(L' - y) > \varphi_L(L - y) \ForAll y \in (0, L' - K').
  \end{equation}
  Finally, $\varphi_{L'}$ is decreasing in $(K', L')$, so \eqref{eq:reverse} implies
  \begin{equation*}
    \varphi_{L'}(y) > \varphi_{L'}(y + L' - L) > \varphi_L(y) \ForAll y \in [K', L].
  \end{equation*}
  This completes the proof of \eqref{eq:increasing}.

  Using this monotonicity, we can establish an \emph{a priori} lower bound for $\varphi_L$.
  Let $\varphi_{0, L}$ denote the solution to \eqref{eq:strip-ODE} with $\varrho = 0$, and let $B_{\trm{ODE}}$ denote the Dirichlet analogue of $A_{\trm{ODE}}$
  Define $v \coloneqq \varphi_{0, B_{\trm{ODE}}}$.
  By extending $\varphi_L$ on the left until it reaches zero, we see that it is simply a shift of $\varphi_{0,L'}$ for some $L' > L$.
  Then \eqref{eq:increasing} and symmetry in $y$ implies that $\varphi_L \geq v$ for all $L > \max\{A_{\trm{ODE}}, B_{\trm{ODE}}\}$.
  Furthermore, if we slide $v$ along $(0, L)$, ODE uniqueness implies that $\varphi_L$ and $v$ cannot touch in the interior, for at the point of first contact they would agree to first order.
  Therefore
  \begin{equation}
    \label{eq:const-bound}
    \varphi_L \geq \max v \quad \textrm{on } \left(\frac{B_{\trm{ODE}}}{2}, \, L - \frac{B_{\trm{ODE}}}{2}\right).
  \end{equation}
  Note that $\max v \in (\vartheta, 1)$.
  \\

  \noindent{\bf Step 5.}
  We can finally prove that $\varphi_L \nearrow \varphi$.
  Since the family $(\varphi_L)_{L > A_{\trm{ODE}}}$ is increasing in $L$, it must have a positive limit.
  By standard elliptic estimates, the convergence is locally uniform and the limit satisfies \eqref{eq:steady-ODE}.
  Thus by Lemma~\ref{lem:steady-ODE-ign-bi}, $\varphi_L \to \varphi$ locally uniformly.
  We prove uniform convergence by contradiction, so suppose there exist $\eps > 0$ and sequences $L_n \nearrow \infty$ and $y_n \in [0, L_n/2]$ such that
  \begin{equation}
    \label{eq:contradiction}
    \varphi_{L_n}(y_n) \leq \varphi(y_n) - \eps.
  \end{equation}
  Then $(y_n)$ cannot have a finite limit point, since $\varphi_L \to \varphi$ locally uniformly.
  Thus $y_n \to \infty$, and there exists $h \in [0,\infty]$ such that $L_n/2 - y_n \to h$ (perhaps after extracting a subsequence).
  Using elliptic estimates and Arzel\`a--Ascoli, $\varphi_{L_n}(\anon - y_n)$ converges locally uniformly (along a subsequence) to some $\phi \colon (-\infty, h] \to \R$.
  Combining \eqref{eq:const-bound} and \eqref{eq:contradiction}, we see that
  \begin{equation*}
    \vartheta < \max v \leq \phi \leq 1 - \eps < 1.
  \end{equation*}
  However, $\phi$ solves $\phi'' + f(\phi) = 0$ and is thus uniformly concave, a contradiction.

  Also, since $s^\al \to \vartheta$ as $\al \searrow 0$, we have
  \begin{equation*}
    \sup_{[0, L]} \psi_L \to \vartheta
  \end{equation*}
  as $L \to \infty$.
  This completes the proof of the lemma.
\end{proof}

Although we do not use the steady states of \emph{monostable} reactions in bounded intervals, we describe them for completeness.
The monostable case is simpler: there is a unique nonzero steady state in large intervals.
\begin{lemma}
  Let $f$ be monostable with $\varrho \in [0, \infty)$.
  Then there exists $A_{\trm{ODE}} > 0$ such that for all $L > A_{\trm{ODE}}$, \eqref{eq:strip-ODE} admits precisely one nonzero solution $\varphi_L$.
  Furthermore, $\varphi_L < \varphi$ and $\varphi_L$ satisfies \eqref{eq:strip-steady-convergence}.
\end{lemma}

\begin{proof}
  We retain the notation of the previous proof.
  The analysis of $\varphi_L$ for ignition and bistable reactions extends to the monostable case.
  The only difference is the behavior of the solutions $\phi^\al$ when $\al \searrow 0$.
  For ignition and bistable reactions, these ``shallow'' solutions extend over arbitrarily large intervals, and form the second solution $\psi_L$.
  However, when $f$ is monostable, we have $\vartheta = 0$, so $s^\al \searrow 0$ as $\al \searrow 0$.
  Since $\phi^\al$ is uniformly small, it is controlled by the linearization of \eqref{eq:strip-ODE} about $0$.
  By \ref{hyp:mono-derivs}, $\phi^\al$ resembles a sine wave of \emph{bounded width}.
  That is, $L^\al$ remains uniformly bounded as $\al \searrow  0$.
  We can thus define
  \begin{equation*}
    A_{\trm{ODE}} \coloneqq \sup_{\al \in (0, \al_1]} L^\al < \infty.
  \end{equation*}
  Any nonzero solution in $[0, L]$ for $L > A_{\trm{ODE}}$ must be $\varphi_L$, proving the lemma.
\end{proof}

We now return to the ignition and bistable setting.
Since \eqref{eq:strip-ODE} has multiple solutions, we are interested in their relative stability.
Following \cite{Vega}, we define the energy
\begin{equation}
  \label{eq:energy}
  \m H(\phi) \coloneqq \int_0^L \left[\abs{\phi'}^2 - 2 F(\phi)\right] \d y \quad \textrm{for } \quad F(s) \coloneqq \int_0^s f(r) \d r,
\end{equation}
and note that the ODE $\phi'' + f(\phi) = 0$ is the Euler--Lagrange equation for $\m H$.
\begin{lemma}
  \label{lem:stability}
  Let $f$ be ignition with $\varrho \in [0, \infty)$ or bistable with $\varrho = 0.$
  Then there exists $A_{\m H} \geq A_{\trm{ODE}}$ such that $\m H(\varphi_L) < \m H(0) < \m H(\psi_L)$ for all $L > A_{\m H}$.
\end{lemma}
Thus, as measured by $\m H$, $\varphi_L$ is the most stable solution of \eqref{eq:strip-ODE} and $\psi_L$ is the least.
In the next section, we use this stability to establish the uniqueness of traveling waves in strips.

\begin{proof}
  Suppose $L > A_{\trm{ODE}}$.
  By Lemma~\ref{lem:strip-ODE-ign-bi}, $\varphi_L$, $0$, and $\psi_L$ are the only solutions to \eqref{eq:strip-ODE}.
  When $L$ is large, $\varphi_L \approx 1$ and $\varphi_L' \approx 0$ on nearly the entire domain.
  By \ref{hyp:ign} or \ref{hyp:bi-pos}, $F(1) > 0$.
  It follows that $\m H(\varphi_L) < 0$ when $L \gg 1$.
  Also, we trivially have $\m H(0) = 0$.
  This leaves only $\m H(\psi_L)$.
  We consider the ignition and bistable cases separately.
  
  First, suppose $f$ is ignition.
  As shown in the proof of Lemma~\ref{lem:strip-ODE-ign-bi}, $\psi_L$ is affine linear outside the interval $[y_1, y_2] \coloneqq \psi_L^{-1}([\theta, 1))$.
  On $[y_1, y_2]$, the function $\psi_L - \theta$ is a Dirichlet solution to \eqref{eq:strip-ODE} with reaction $f(\anon + \theta)$, which is monostable on the interval $[0, 1 - \theta]$.
  By \ref{hyp:ign-derivs}, $\psi_L - \theta$ resembles a small sine wave of bounded width when $L$ is large.
  Precisely,
  \begin{equation*}
    \psi_L\left(y\right) - \theta \sim \frac{2\theta}{L} \sin\left[\sqrt{f'(\theta^+)} (y - y_1)\right] \quad \textrm{on } [y_1, y_2]
  \end{equation*}
  as $L \to \infty$.
  Now, $F(s - \theta) \sim \frac{1}{2} f'(\theta^+) s^2$ as $s \searrow 0$, so
  \begin{equation*}
    \int_{y_1}^{y_2} F(\psi_L(y)) \d y \leq C L^{-2}
  \end{equation*}
  for some constant $C \geq 1$ which may change from line to line.
  Now, $F(\psi_L)$ vanishes outside $[y_1, y_2]$, where $|\psi_L'| \sim 2 \theta L^{-1}$.
  Therefore
  \begin{equation*}
    \m H(\psi_L) \geq \int_{[0, L] \setminus [y_1, y_2]} |\psi_L'|^2 - 2\int_{y_1}^{y_2} F(\psi_L) \geq C^{-1} L^{-1} - C L^{-2}.
  \end{equation*}
  It follows that $\m H(\psi_L) > 0$ once $L$ is sufficiently large.

  Now, suppose $f$ is bistable.
  Then $F$ is negative below $\vartheta$ and positive above it.
  We view $\psi_L$ about its maximum at $y = L/2$.
  As $L \to \infty$, it converges uniformly to a positive soliton in $\R$ solving
  \begin{equation*}
    \psi'' + f(\psi) = 0, \quad \psi(0) = \vartheta, \quad \psi'(0) = 0.
  \end{equation*}
  Since $0 < \psi \leq \vartheta$, $\psi$ has positive energy.
  Clearly $\m H(\psi_L)$ converges to this energy, so $\m H(\psi_L) > 0$ for $L$ sufficiently large.
  
  Thus in each case, there exists $A_{\m H} \geq A_{\trm{ODE}}$ such that $\m H(\varphi_L) < \m H(0) < \m H(\psi_L)$ when $L > A_{\m H}.$
\end{proof}

\section{Traveling waves in strips}
\label{sec:strip-TW}

In this section, we construct and control traveling waves in strips $\R \times [0, L]$.
We thus work solely in two dimensions.
We will always assume that $L > A_{\trm{ODE}}$, so that Lemma~\ref{lem:strip-ODE-ign-bi} classifies the steady-state solutions of \eqref{eq:strip-ODE}.

\subsection{Construction}

We wish to construct solutions to the elliptic boundary-value problem
\begin{equation}
  \label{eq:strip-TW}
  \begin{cases}
    \Delta \Phi_L + c_L \partial_x \Phi_L + f(\Phi_L) = 0 & \textrm{in } \R \times (0, L),\\
    \partial_y \Phi_L(x, 0) = \varrho^{-1} \Phi_L(x, 0),\\
    \Phi_L(x, L) = 0,\\
    \Phi_L(-\infty, y) = \varphi_L(y),\\
    \Phi_L(+\infty, y) = 0
  \end{cases}
\end{equation}
for some speed $c_L \in \R$ to be determined.
The wave $\Phi_L$ satisfies our usual absorbing boundary condition on the lower boundary $\R \times \{0\}$ and a Dirichlet condition on the upper boundary $\R \times \{L\}$.
Moreover, it connects the steady states $\varphi_L$ and $0$ from left to right.

To begin, we construct approximate traveling waves in the rectangle
\begin{equation*}
  \Omega_a \coloneqq  (-a, a) \times (0, L)
\end{equation*}
for $a > 0$.
We study $\Phi^{a, c}$ in $\bar \Omega_a$ satisfying
\begin{equation}
  \label{eq:TW-box-ign-bi}
  \begin{cases}
    \Delta \Phi^{a, c} + c \partial_x \Phi^{a, c} + f(\Phi^{a, c}) = 0 & \textrm{in } \Omega_a,\\
    \partial_y\Phi^{a, c}(x, 0) = \varrho^{-1}\Phi^{a, c}(x, 0),\\
    \Phi^{a, c}(x, L) = 0,\\
    \Phi^{a, c}(-a, y) = \varphi_L(y),\\
    \Phi^{a, c}(a, y) = 0.
  \end{cases}
\end{equation}
For the moment, we treat $a > 0$ and $c \in \R$ as fixed.
Of course, $\Phi^{a, c}$ also depends on $L$, but we suppress this dependence for the sake of legibility.

\begin{lemma}
  \label{lem:box-ign-bi}
  Let $f$ be ignition with $\varrho \in [0, \infty)$ or bistable with $\varrho = 0.$
  Let $L > A_{\trm{ODE}}$ and $a > 0$.
  Then for each $c \in \R$, there exists a unique solution $\Phi^{a, c}$ to \eqref{eq:TW-box-ign-bi}.
  Moreover, $0 < \Phi^{a, c} < \varphi_L$ and $\partial_x\Phi^{a, c} < 0$ in $\Omega_a$.
  If $\varrho = 0$, $\Phi^{a, c}$ is even in $y$ about $y = L/2$ and $\partial_y \Phi^{a, c} > 0$ on $(-a, a) \times (0, L/2)$.
\end{lemma}

\begin{proof}
  For existence, we observe that $0$ and $\varphi_L$ form an ordered pair of sub- and supersolutions, respectively, for \eqref{eq:TW-box-ign-bi}.
  Hence if we solve the parabolic version of \eqref{eq:TW-box-ign-bi} with initial data $\varphi_L$, the solution will converge monotonically to a steady state $\Phi$ between $0$ and $\varphi_L$ as $t \to \infty$.
  Also, the parabolic evolution from constant initial data shows that any solution is confined between the maximal and minimal solutions of \eqref{eq:strip-ODE}, namely $\varphi_L$ and $0$.

  In our uniqueness argument, Robin and Dirichlet boundaries require superficially different treatment.
  We thus first suppose that $\varrho > 0$.
  We claim that
  \begin{align}
    \label{eq:Robin-lower}
    \Phi &> 0 \quad \textrm{in } [-a,a) \times [0, L),\\
    \label{eq:deriv-lower}
    \partial_y \Phi &< 0 \quad \textrm{on } [-a, a) \times \{L\}.
  \end{align}
  Since $\varphi_L > 0$ on $[0, L)$, \eqref{eq:Robin-lower} holds on the left boundary.
  By the strong maximum principle, it extends to the interior $(-a, a) \times (0, L)$.
  Then the Hopf lemma implies that $\Phi$ and $\partial_y \Phi$ cannot simultaneously vanish on $(-a, a) \times \{0\}$.
  By the Robin condition, $\Phi > 0$ there; \eqref{eq:Robin-lower} follows.
  Of course, we cannot extend \eqref{eq:Robin-lower} to the upper boundary, but the Hopf lemma yields \eqref{eq:deriv-lower}.
  By identical reasoning, we obtain
  \begin{align}
    \label{eq:Robin-upper}
    \Phi &< \varphi_L \quad \textrm{in } (-a, a] \times [0, L),\\
    \label{eq:deriv-upper}
    \partial_y \Phi &> \varphi_L'(L) \quad \textrm{on } (-a, a] \times \{L\}.
  \end{align}
  
  We employ these bounds in the sliding method to establish uniqueness and monotonicity.
  Suppose $\ti \Phi$ is another solution to \eqref{eq:TW-box-ign-bi}.
  As argued above, ${0 \leq \ti \Phi \leq \varphi_L.}$
  Also, $\ti \Phi$ satisfies \eqref{eq:Robin-lower}--\eqref{eq:deriv-upper}.
  For $\ell \in [0, 2a]$, consider the shifted function
  \begin{equation*}
    \ti \Phi_\ell (x, y) \coloneqq \ti \Phi (x + \ell, y)
  \end{equation*}
  in $\Omega_{a}^\ell \coloneqq \Omega_{a} - \ell \tbf{e}_x$ and let
  \begin{equation*}
    D^\ell \coloneqq \Omega_{a} \cap \Omega_{a}^\ell = (- a, a - \ell) \times (0,L).
  \end{equation*}
  Define
  \begin{equation*}
    \sigma \coloneqq \inf\big\{\ell \in [0, 2a] \mid \ti \Phi_\ell \leq \Phi\big\},
  \end{equation*}
  noting that $\ti \Phi_{2a} \leq \Phi$ because $\varphi_L \geq 0$, so $\sigma \in [0, 2a]$.
  Suppose for the sake of contradiction that $\sigma > 0$.
  By continuity, $\ti \Phi_\sigma \leq \Phi$ in $\bar D^\sigma$.
  We claim that
  \begin{align}
    \label{eq:solution-comparison}
    \ti \Phi_\sigma &< \Phi \quad \textrm{on } [-a,a - \sigma] \times [0, L),\\
    \label{eq:deriv-solution-comparison}
    \partial_y \ti \Phi_\sigma &> \partial_y \Phi \quad \textrm{on } [-a, a - \sigma] \times \{L\}.
  \end{align}
  
  Since $\Phi$ satisfies \eqref{eq:Robin-lower} and $\ti \Phi$ satisfies \eqref{eq:Robin-upper}, we automatically have $\ti \Phi_\sigma < \Phi$ on the vertical sides of $D^\sigma$.
  The strong maximum principle extends strict inequality to the interior $D^\sigma$.
  Then the Robin condition and the Hopf lemma force $\ti \Phi_\sigma < \Phi$ on the lower side, and we obtain \eqref{eq:solution-comparison}.
  Again, the Hopf lemma implies \eqref{eq:deriv-solution-comparison}.

  By elliptic estimates, our solutions are $\m C^1$.
  It follows that \eqref{eq:solution-comparison} and \eqref{eq:deriv-solution-comparison} are open conditions on $\sigma$.
  That is, \eqref{eq:solution-comparison} and \eqref{eq:deriv-solution-comparison} imply
  \begin{align*}
    \ti \Phi_\ell &< \Phi \quad \textrm{on } [-a,a - \ell] \times [0, L),\\
    \partial_y \ti \Phi_\ell &> \partial_y \Phi \quad \textrm{on } [-a, a - \ell] \times \{L\}
  \end{align*}
  for all $\ell$ sufficiently close to $\sigma$.
  This contradicts the definition of $\sigma$, so in fact $\sigma = 0$.

  We have shown that $\ti \Phi \leq \Phi$.
  Reversing the roles of $\Phi$ and $\ti \Phi$, we see that $\Phi$ is unique.
  Furthermore, $\Phi_\ell \leq \Phi$ for all $\ell \in [0, 2a]$, so $\partial_x \Phi\leq 0$.
  Since $\Phi$ is not constant in $x$, the strong maximum principle implies that $\partial_x \Phi < 0$ in $\Omega_a$.

  We now turn to the Dirichlet case $\varrho = 0$.
  Then, $0, \varphi_L$, and $\Phi$ agree on the lower boundary, so we must include derivative conditions there as well.
  The proof is otherwise identical, so we do not repeat it.
  We do, however, study the $y$-dependence of $\Phi$.
  Since the upper and lower boundary conditions agree, $\Phi(x, L - y)$ is also a solution.
  By uniqueness, $\Phi(x, L - y) = \Phi(x, y)$, as claimed.
  
  If we restrict our solution to $[-a, a] \times [0, L/2]$, symmetry implies that $\partial_y \Phi = 0$ on $[-a, a] \times \{L/2\}$.
  That is, $\Phi$ satisfies a Neumann condition on this new upper boundary.
  The sliding argument above works with Neumann conditions as well, so there is no other solution in the lower half-box.
  Since $\varphi_L$ is increasing on $[0, L/2]$ (by ODE uniqueness), the boundary conditions are monotone in $y$.
  As the limit of the parabolic evolution from $0$ in the half-box, $\Phi$ is nondecreasing in $y$ when $y \leq L/2$.
  By the strong maximum principle, $\partial_y \Phi > 0$ on $(-a, a) \times (0, L/2)$.

  This concludes the proof of the lemma.
  In the remainder of the paper, we denote the unique solution $\Phi$ by $\Phi^{a, c}$.
\end{proof}

We are interested in the dependence of $\Phi^{a, c}$ on $c$.
From the previous lemma, we immediately obtain:
\begin{corollary}
  \label{cor:monotone-in-c}
  For each $L > A_{\trm{ODE}}$ and $a > 0$, the solution $\Phi^{a, c}$ to \eqref{eq:TW-box-ign-bi} is continuous and decreasing in $c$.
\end{corollary}

\begin{proof}
  Continuity follows from standard elliptic estimates.
  Suppose $c_1 < c_2$.
  By Lemma~\ref{lem:box-ign-bi}, $(c_2 - c_1) \partial_x \Phi^{a, c_2} < 0$.
  Hence $\Phi^{a, c_2}$ is a subsolution to \eqref{eq:TW-box-ign-bi} for $c = c_1$.
  Since the solutions are unique, $\Phi^{a, c_2} \leq \Phi^{a, c_1}$.
\end{proof}

Presumably,
\begin{equation*}
  \lim_{c \to -\infty} \Phi^{a, c} = \varphi_L \And \lim_{c \to \infty} \Phi^{a, c} = 0
\end{equation*}
locally uniformly in $\Omega_a$.
We do not prove this---weaker bounds suffice for our purposes.
Recall that $c_*$ denotes the unique speed of the one-dimensional wave for $f$ connecting $1$ to $0$.
\begin{lemma}
  \label{lem:upper}
  For every $s \in (0, 1)$ and $L > A_{\trm{ODE}}$, there exists $A_{\trm{upper}}(s) > 0$ such that $\Phi^{a, c_*}(0, L/2) < s$ for all $a > A_{\trm{upper}}$.
\end{lemma}

\begin{proof}
  Let $U$ denote the one-dimensional traveling wave, so that
  \begin{equation*}
    U'' + c_* U' + f(U) = 0, \quad U(-\infty) = 1, \And U(+\infty) = 0.
  \end{equation*}
  Of course, this only determines $U$ up to translation, so we further assume that $U(0) = \vartheta$.

  Let $B \coloneqq U^{-1}(\max \varphi_L) \in \R$.
  Then $U(\anon + a + B)$ is a supersolution to \eqref{eq:TW-box-ign-bi} when $c = c_*$.
  Since the solution to \eqref{eq:TW-box-ign-bi} is unique, this implies that
  \begin{equation*}
    \Phi^{a, c_*}(x, y) \leq U(x + a + B) \ForAll (x, y) \in \bar{\Omega}_a.
  \end{equation*}
  In particular,
  \begin{equation*}
    \Phi^{a, c_*}(0, L/2) \leq U(a + B).
  \end{equation*}
  Now $U(+\infty) = 0$, so for each $s \in (0, 1)$ there exists $A_{\trm{upper}}(s) > 0$ such that ${U(A_{\trm{upper}} + B) \leq s.}$
  Then $\Phi^{a, c_*}(0, L/2) < s$ for all $a > A_{\trm{upper}}$, as $U$ is monotone decreasing.
\end{proof}

\begin{lemma}
  \label{lem:lower}
  For all $s \in (0, 1)$, there exist $A_{\trm{lower}}(s) \geq A_{\trm{ODE}}$ and $\gamma(s) > 0$ such that the following holds.
  For all $a, L > A_{\trm{lower}}$, we have $\Phi^{a, \gamma}(0, L/2) > s.$
\end{lemma}

\begin{proof}
  We wish to construct a compactly supported subsolution which attains the value $s$.
  It thus suffices to consider the Dirichlet case $\varrho = 0$ and $s \in (\vartheta, 1)$.

  We follow the construction of subsolutions in the proof of Theorem~\ref{thm:steady}\ref{item:steady-ign-bi}.
  There, we used $\mr{\phi}$ solving
  \begin{equation*}
    \mr{\phi}'' + f(\mr{\phi}) = 0, \quad \mr{\phi}(0) = s, \quad \mr{\phi}'(0) = 0,
  \end{equation*}
  which hits zero at some position $\mr K > 0$.
  To adapt this to two dimensions with a small drift, we consider
  \begin{equation*}
    \ti \phi'' + \left(\gamma + \frac{1}{R_0 + y}\right) \ti \phi' +  f(\ti \phi) = 0, \quad \ti \phi(0) = s, \quad \ti \phi'(0) = 0.
  \end{equation*}
  By the stability of ODEs, there exist $\gamma, R_0 > 0$ such that $\ti \phi$ still hits $0$ at some position $K_0 > 0$.

  Using $\ti \phi$, we construct a radial subsolution $v$.
  Let $r \coloneqq |\tbf{x} - L\tbf{e}_y/2|$ denote the radial coordinate centered at $(0, L/2)$, and define
  \begin{equation*}
    v(r) \coloneqq
    \begin{cases}
      s & \textrm{for } r \leq R_0,\\
      \ti\phi\big(r - R_0\big) & \textrm{for } r \in (R_0, R_0 + K_0),\\
      0 & \textrm{for } r \geq R_0 + K_0.
    \end{cases}
  \end{equation*}

  Let $A_{\trm{lower}} \coloneqq \max\{2(R_0 + K_0), A_{\trm{ODE}}\}$.
  Then $v$ is a compactly supported subsolution to \eqref{eq:TW-box-ign-bi} when $a, L > A_{\trm{lower}}$ and $c \in [0, \gamma]$.
  Since $\Phi^{a, c}$ is the unique solution to \eqref{eq:TW-box-ign-bi}, we have $\Phi^{a, c} > v$.
  In particular, $\Phi^{a, \gamma}(0, L/2) > v(0, L/2) = s$.
\end{proof}

Now fix $\theta_0 \in (\vartheta, 1)$ and define $\gamma \coloneqq \gamma(\theta_0)$ and
\begin{equation*}
  A_{\trm{pin}} \coloneqq \max\left\{A_{\trm{upper}}(\theta_0), A_{\trm{lower}}(\theta_0)\right\}.
\end{equation*}
Then Lemmas~\ref{lem:upper} and \ref{lem:lower} imply that
\begin{equation*}
  \Phi^{a, c_*}(0, L/2) < \theta_0 < \Phi^{a, \gamma}(0, L/2)
\end{equation*}
for all $a, L > A_{\trm{pin}}$.
By Corollary~\ref{cor:monotone-in-c}, $\Phi^{a, c}(0, L/2)$ is monotone and continuous in $c$.
Hence, there exists $c^a \in (\gamma, c_*)$ such that
\begin{equation}
  \label{eq:pin}
  \Phi^{a, c^a}(0, L/2) = \theta_0.
\end{equation}

We now take $a \to \infty$.
By standard elliptic estimates, there exists a locally uniform subsequential limit $(\Phi_L, c_L)$ of $\big(\Phi^{a, c^a}, c^a\big)$ for each $L > A_{\trm{pin}}$.
\begin{proposition}
  \label{prop:strip-TW}
  Let $f$ be ignition with $\varrho \in [0, \infty)$ or bistable with $\varrho = 0$.
  Then there exists $A_{\trm{TW}} \geq \max\big\{A_{\trm{pin}}, A_{\m H}\big\}$ such that for all $L > A_{\trm{TW}},$ the subsequential limit $(\Phi_L, c_L)$ solves \eqref{eq:strip-TW} and satisfies $\partial_x \Phi_L < 0$ and $c_L \in [\gamma, c_*]$.
  Moreover, if $\varrho = 0$, then $\Phi_L$ is symmetric in $y$ about $y = L/2$ and $\partial_y \Phi_L > 0$ on $\R \times (0, L/2)$.
\end{proposition}

\begin{proof}
  By elliptic regularity and Lemma~\ref{lem:box-ign-bi}, the only question is the limiting behavior as $x \to \pm \infty$.
  The wave $\Phi_L$ is monotone decreasing in $x$, so these limits exist and solve the steady-state equation \eqref{eq:strip-ODE}.
  The choice \eqref{eq:pin} ensures that $\Phi_L(0, L/2) = \theta_0$.
  Furthermore, by Lemma~\ref{lem:strip-ODE-ign-bi}, there exists $A_{\trm{TW}} \geq \max\big\{A_{\trm{pin}}, A_{\m H}\big\}$ such that $\psi_L(L/2) < \theta_0 < \varphi_L(L/2)$ for all $L > A_{\trm{TW}}$.
  It follows that ${\Phi_L(-\infty, \anon) = \varphi_L}$.
  However, the right limit is still in doubt: $\Phi_L(+\infty, \anon) = \psi_L$ or $0$.
  This is one of the most delicate issues in the paper.

  Before grappling with the question directly, we define a reduced reaction $\ti{f}$.
  Let $\eta \coloneqq \max \psi_L - \vartheta > 0$.
  By Lemma~\ref{lem:strip-ODE-ign-bi}, we can make $\eta$ arbitrarily small by taking $L$ large.
  Increasing $A_{\trm{TW}}$ if need be, we can assume that $\ti \theta \coloneqq \vartheta + 2\eta < 1.$
  Now set $\ti{f} = f$ on $[0, \ti{\theta} - \eta/3]$ and let $\ti f$ smoothly connect to $0$ on $[\ti{\theta} - \eta/3, \ti{\theta}]$ while remaining below $f$.
  Then $\ti{f} \leq f$ is an ignition or bistable reaction (in accordance with $f$) on the restricted interval $[0, \ti{\theta}]$.
  Let $\ti{c}$ denote its one-dimensional wave speed with wave $\ti{U}$ connecting $\ti{\theta}$ to $0$, so that
  \begin{equation*}
    \ti U'' + \ti c \ti U' + \ti f\big(\ti U\big) = 0, \quad \ti U(-\infty) = \ti \theta, \And \ti U(+\infty) = 0.
  \end{equation*}
  We are free to fix the translation of $\ti{U}$ so that
  \begin{equation*}
    \ti{U}(0) = \max \psi_L + \frac{2\eta}{3} = \ti{\theta} - \frac{\eta}{3}.
  \end{equation*}

  We wish to use $\ti{U}$ as a supersolution to force $\Phi_L$ to converge to $0$ as ${x \to + \infty}$.
  However, traveling waves are only supersolutions when viewed in a \emph{faster} frame.
  We must therefore arrange $\ti c \leq c_L$.
  To do so, recall that $\ti f, \ti \theta$, and $\ti U$ depend implicitly on $L$ through the small parameter $\eta = \max \psi_L - \vartheta$.
  By Lemma~\ref{lem:strip-ODE-ign-bi}, $\eta \to 0$ as $L \to \infty$.
  Then $\ti \theta \to \vartheta$ and
  \begin{equation*}
    \ti f(s) \to \tbf{1}_{[0, \vartheta]}(s) f(s) \ForAll s \in [0, 1].
  \end{equation*}
  Recall that $\int_0^\vartheta f = 0$.
  This implies that the ``traveling wave'' with reaction $\tbf{1}_{[0, \vartheta]} f$ connecting $\vartheta$ to $0$ is actually stationary.
  By the continuity of one-dimensional waves, we obtain $\ti c \to 0$ as $L \to \infty$.
  In particular, increasing $A_{\trm{TW}}$ if need be, we can assume that $\ti c \leq \gamma \leq c_L$.

  We can now control our traveling wave $\Phi_L$.
  As shown earlier, the right limit $\Phi_L(+\infty, \anon)$ is $\psi_L$ or $0$.
  By Dini's theorem, the convergence is uniform.
  Hence, there exists $B \in \R$ such that
  \begin{equation}
    \label{eq:Dini}
    \sup_{y \in [0, L]} \left[\Phi_L(B, y) - \psi_L(y)\right] < \frac{\eta}{3}.
  \end{equation}
  Now, the convergence $\Phi^{a, c^a} \to \Phi_L$ as $a \to \infty$ is locally uniform (along a subsequence, which we suppress for clarity).
  Thus, there exists $A > B$ such that
  \begin{equation}
    \label{eq:locally-uniform}
    \sup_{y \in [0, \, L]}\abs{\Phi^{a, c^a}(B, y) - \Phi_L(B, y)} < \frac{\eta}{3}
  \end{equation}
  when $a > A$.
  Combining \eqref{eq:Dini} and \eqref{eq:locally-uniform}, the triangle inequality yields
  \begin{equation*}
    \sup_{y \in [0, \, L]} \Phi^{a, c^a}(B, y) < \max \psi_L + \frac{2\eta}{3} = \ti{U}(0).
  \end{equation*}
  Since $\partial_x \Phi^{a, c^a} < 0$, this implies that
  \begin{equation}
    \label{eq:right-bound}
    \sup_{[B, \, a] \times [0, \, L]} \Phi^{a, c^a} < \ti{U}(0).
  \end{equation}
  By construction, $f = \ti{f}$ on the interval $\big[0, \ti{U}(0)\big]$.
  Thus, the approximate wave $\Phi^{a, c^a}$ satisfies
  \begin{equation*}
    \Delta \Phi^{a, c^a} + c^a \partial_x \Phi^{a, c^a} + \ti f\big(\Phi^{a, c^a}\big) = 0 \quad \textrm{in } (B, a) \times (0, L).
  \end{equation*}
  Moreover, $\ti{c} \leq \gamma \leq c^a$ and $\ti U' < 0$ imply that $\ti{U}$ is a supersolution to this equation.

  Since $\ti U'$ is decreasing, \eqref{eq:right-bound} implies that
  \begin{equation}
    \label{eq:wave-sliding}
    \ti{U}(x - x_a) \geq \Phi^{a, c^a}(x, y) \ForAll (x, y) \in [B, a] \times [0, L]
  \end{equation}
  and each $x_a \geq a \geq A$.
  We now use the sliding method to continuously reduce the shift $x_a$ to $B$.
  When ${x_a > B}$, \eqref{eq:right-bound} and the Hopf lemma imply that $U(\anon - x_a)$ cannot touch $\Phi^{a, c^a}$ on the boundary of $[B, a] \times [0, L]$.
  Moreover, the sliding supersolution cannot touch the solution in the interior, by the strong maximum principle.
  Hence \eqref{eq:wave-sliding} holds for all $x_a \geq B$.
  In particular,
  \begin{equation*}
    \ti{U}(x - B) \geq \Phi^{a, c^a}(x, y) \ForAll (x, y) \in [B, a] \times [0, L]
  \end{equation*}
  whenever $a \geq A$.
  We emphasize that the left side is independent of $a$.
  It follows that
  \begin{equation*}
    \ti{U}(x - B) \geq \Phi_L(x, y) = \lim_{a \to \infty} \Phi^{a, c^a}(x, y) \ForAll (x, y) \in [B, \infty) \times [0, L].
  \end{equation*}
  Of course, $\ti{U}(+\infty) = 0$, so $\Phi_L(+\infty, \anon) = 0$ as desired.
\end{proof}

\subsection{Properties}

Now that we have a traveling wave, results of Vega \cite{Vega} imply uniqueness.
\begin{lemma}
  \label{lem:TW-unique}
  Let $f$ be ignition with $\varrho \in [0, \infty)$ or bistable with $\varrho = 0$.
  Then for all $L > A_{\trm{TW}}$, there is a unique speed $c_L$ such that \eqref{eq:strip-TW} admits a solution $\Phi_L$ which is monotone in $x$.
  Furthermore, $\Phi_L$ is unique up to translation.
\end{lemma}

\begin{proof}
  We use the results of \cite{Vega}.
  There, the behavior of waves hinges on the stability of the limiting steady states, as measured by the energy $\m H$ defined in \eqref{eq:energy}.
  Since $L > A_{\trm{TW}} \geq A_{\trm{ODE}}$, the only solutions of \eqref{eq:strip-TW} are $\varphi_L, \psi_L$, and $0$.
  By Lemma~\ref{lem:stability} and $L > A_{\trm{TW}} \geq A_{\m H}$, they satisfy
  \begin{equation*}
    \m H(\varphi_L) < \m H(0) < \m H(\psi_L).
  \end{equation*}
  By Theorems 5.1 and 5.2 in \cite{Vega}, \eqref{eq:strip-TW} admits a monotone solution $\Phi_L$ at a unique speed $c_L$, and $\Phi_L$ is unique up to translation.
  We note that \cite{Vega} assumes for convenience that $\Phi_L(-\infty, \anon) < \Phi_L(+\infty, \anon)$.
  We can easily arrange this by replacing $\Phi_L$ by $-\Phi_L$ and $f(s)$ by $-f(-s)$.
  Also, \cite{Vega} only handles Dirichlet conditions.
  However, as Vega notes, the proofs extend to Robin conditions without change.
\end{proof}

Next, we show that $c_L$ converges to the one-dimensional wave speed $c_*$ as ${L \to \infty}$.
\begin{lemma}
  \label{lem:speed-convergence}
  Let $f$ be ignition with $\varrho \in [0, \infty)$ or bistable with $\varrho = 0$.
  Then $c_L \to c_*$ as $L \to \infty$.
\end{lemma}

\begin{proof}
  Define the recentered wave
  \begin{equation*}
    \Psi_L(x, y) \coloneqq \Phi_L\Big(x, y + \frac{L}{2}\Big).
  \end{equation*}
  Then \eqref{eq:pin} and Proposition~\ref{prop:strip-TW} imply that $\Psi_L(0, 0) = \theta_0$, $0 < \Psi_L < \varphi_L(\anon + L/2)$, and $\partial_x \Psi_L < 0$.
  If $\varrho = 0$, $\Psi_L$ is also even in $y$ and $\partial_y \Psi_L > 0$ when $y < 0$.

  Now let $L_k \to \infty$ be a sequence of lengths such that
  \begin{equation*}
    c_{L_k} \to c_- \coloneqq \liminf_{L \to \infty} c_L.
  \end{equation*}
  Note that, by Proposition~\ref{prop:strip-TW}, $c_- \in [\gamma, c_*]$.
  By standard elliptic estimates, there exists a subsequence (which we also call $L_k$ for simplicity) such that $\Psi_{L_k}$ has a locally uniform limit $\Psi$ in $\R^2.$
  Furthermore, $\Psi$ satisfies
  \begin{equation*}
    \Delta \Psi + c_- \partial_x \Psi + f(\Psi) = 0
  \end{equation*}
  as well as $\Psi(0, 0) = \theta_0$, $0 < \Psi < 1$, and $\partial_x \Psi < 0$.
  Moreover, if $\varrho = 0$, $\Psi$ is even in $y$ and $\partial_y \Psi > 0$ when $y < 0$.

  By the monotonicity in $x$, the limits $\Psi(\pm \infty, \anon)$ exist.
  Moreover, the convergence is locally uniform and the limits satisfy the steady-state equation $\phi'' + f(\phi) = 0$ in $\R$.
  The only solutions to this equation are constant or non-constant periodic, and the latter is only possible when $f$ is bistable.
  However, $\varrho = 0$ when $f$ is bistable.
  In this case, the limits $\Psi(\pm \infty, \anon)$ are also even and monotone increasing on $\R_-$, and thus must be constant.

  It follows that $\Psi(-\infty, \anon) = s_-$ for some $s_- \in [0, 1]$ such that $f(s_-) = 0$.
  Furthermore, $s_- > \Psi(0, 0) = \theta_0 > \vartheta$.
  The only zero of $f$ above $\vartheta$ is $1$, so $\Psi(-\infty, \anon) = 1$.
  Hence for any $R, \delta > 0$, there exists a radius-$R$ ball $B \subset \R^2$ such that ${\Psi|_B \geq 1 - \delta}$.
  Now let $w$ solve
  \begin{equation*}
    \begin{cases}
      \partial_t w = \Delta w + f(w),\\
      w(0, \tbf{x}) = (1 - \delta) \tbf{1}_B(\tbf{x})
    \end{cases}
  \end{equation*}
  in $\R^2$.
  Then the comparison principle implies that
  \begin{equation}
    \label{eq:whole-space-comparison}
    w(t, x, y) \leq \Psi(x - c_-t, y) \ForAll (t, x, y) \in [0, \infty) \times \R^2.
  \end{equation}
  However, in the whole space, Aronson and Weinberger \cite{AW} show that $w \to 1$ locally uniformly, provided $\delta \ll 1$ and $R \gg 1$.
  Furthermore, this disturbance propagates at the asymptotic speed $c_*$.
  Since $\Psi \not \equiv 1$, this will contradict \eqref{eq:whole-space-comparison} unless $c_- \geq c_*$.
  But $c_- \in [\gamma, c_*]$, so in fact $c_- = c_*.$
  Now
  \begin{equation*}
    c_* = \liminf_{L \to \infty} c_L \leq \limsup_{L \to \infty} c_L \leq c_*,
  \end{equation*}
  so $\lim_{L \to \infty} c_L = c_*$, as desired.
\end{proof}

\subsection{Modified traveling waves}

To close this section, we discuss a few variations on our traveling waves.

Throughout, we have exploited symmetry and monotonicity in $y$ when $\varrho = 0$.
When $f$ is ignition and $\varrho > 0$, our waves $\Phi_L$ are more complicated, since they involve Robin conditions on one boundary and Dirichlet conditions on the other.
In our proof of Theorem~\ref{thm:TW}\ref{item:TW-ign-bi}, we will need traveling waves for $\varrho > 0$ which are symmetric in $y$.
This is a simple matter of changing the upper boundary condition.
We therefore consider symmetric steady states satisfying
\begin{equation}
  \label{eq:strip-ODE-sym}
  \begin{gathered}
    (\phi_L^{\mathrm{sym}})'' + f(\phi_L^{\mathrm{sym}}) = 0,\\
    (\phi_L^{\mathrm{sym}})'(0) = \varrho^{-1} \phi_L^{\mathrm{sym}}(0), \quad (\phi_L^{\mathrm{sym}})'(L) = -\varrho^{-1} \phi_L^{\mathrm{sym}}(L).
  \end{gathered}
\end{equation}
These behave exactly like the asymmetric steady states.
\begin{lemma}
  Let $f$ be ignition with $\varrho > 0$.
  Then there exists $A_{\trm{ODE}}^{\mathrm{sym}} > 0$ such that  \eqref{eq:strip-ODE-sym} admits precisely two nonzero solutions $\varphi_L^{\mathrm{sym}} > \psi_L^{\mathrm{sym}}$ when $L > A_{\trm{ODE}}^{\mathrm{sym}}$.
  Furthermore, these solutions satisfy \eqref{eq:strip-steady-convergence} and \eqref{eq:strip-psi-vartheta}.
\end{lemma}

\begin{proof}
  The proof of Lemma~\ref{lem:strip-ODE-ign-bi} extends to this setting.
\end{proof}

We then need traveling waves $\Phi_L^{\mathrm{sym}}$ satisfying
\begin{equation}
  \label{eq:strip-TW-sym}
  \begin{cases}
    \Delta \Phi_L^{\mathrm{sym}} + c_L^{\mathrm{sym}} \partial_x \Phi_L^{\mathrm{sym}} + f(\Phi_L^{\mathrm{sym}}) = 0 & \textrm{in } \R \times (0, L),\\
    \partial_\nu \Phi_L^{\mathrm{sym}} = \varrho^{-1} \Phi_L^{\mathrm{sym}} & \textrm{on } \R \times \{0, L\},\\
    \Phi_L^{\mathrm{sym}}(-\infty, y) = \varphi_L^{\mathrm{sym}}(y),\\
    \Phi_L^{\mathrm{sym}}(+\infty, y) = 0,
  \end{cases}
\end{equation}
where $\partial_\nu$ denotes the partial derivative with respect to the inward normal on the boundary $\R \times \{0, L\}$.
We collect our traveling wave results into one statement.
\begin{proposition}
  \label{prop:omnibus-TW-ign-sym}
  Let $f$ be ignition with $\varrho > 0$.
  Then there exists $A_{\trm{TW}}^{\mathrm{sym}} \geq A_{\trm{ODE}}^{\mathrm{sym}}$ such that for all $L > A_{\trm{TW}}^{\mathrm{sym}}$, the following holds.
  There exists a monotone solution $(\Phi_L^{\mathrm{sym}}, c_L^{\mathrm{sym}})$ to \eqref{eq:strip-TW-sym} which is unique up to translation.
  Moreover, $c_L^{\mathrm{sym}} \in [\gamma, c_*]$, $\partial_x \Phi_L^{\mathrm{sym}} < 0$, $\Phi_L^{\mathrm{sym}}$ is symmetric in $y$ about $y = L/2$, and $\partial_y \Phi_L^{\mathrm{sym}} > 0$ on $\R \times (0, L/2)$.
  Finally, $c_L^{\mathrm{sym}} \to c_*$ as $L \to \infty$.
\end{proposition}

\begin{proof}
  The arguments in the preceding subsections extend to this setting.
\end{proof}

Now, we turn to a different variation on $\Phi_L$.
Given $0 \leq \eps \ll 1$, we define an $\eps$-modification $\ubar{f}$ of $f$.
If $f$ is ignition, let $\ubar{f} = f$ and simply view $f$ as an ignition reaction on the larger interval $[-\eps, 1]$.
If $f$ is bistable, we continuously reduce $f$ on $[-\eps, \eps]$ so that $\ubar{f} \leq f$, $\ubar{f}$ is bistable on $[-\eps, 1]$, and $\|\ubar{f} - f\|_\infty < \eps^2$.
In particular, we can assume that \ref{hyp:bi-pos} continues to hold for $\ubar{f}$, provided $\eps$ is sufficiently small.

When we say that $\ubar{f}$ is ignition or bistable on $[-\eps, 1]$, we mean that $\ubar{f}$ satisfies \ref{hyp:ign} and \ref{hyp:ign-derivs} or \ref{hyp:bi} and \ref{hyp:bi-derivs} after we apply the transformation $s \mapsto \frac{s + \eps}{1 + \eps}.$
We can thus apply all our above results to $\ubar{f}$, provided we interpret them properly.
In particular, we must adjust our boundary conditions.
For instance, by Lemma~\ref{lem:steady-ODE-ign-bi}, $\ubar{f}$ has a unique nonzero bounded steady state $\ubar{\varphi}$ in $\R_+$ satisfying
\begin{equation*}
  \ubar{\varphi}'' + \ubar{f}(\ubar{\varphi}) = 0, \quad \ubar{\varphi}'(0) = \varrho^{-1} \left[\ubar{\varphi}(0) + \eps\right].
\end{equation*}
On sufficiently large intervals, Lemma~\ref{lem:strip-ODE-ign-bi} provides two nonzero solutions ${\ubar \varphi_L > \ubar \psi_L}$ to the ODE
\begin{equation}
  \label{eq:eps-strip-ODE}
  \ubar{\phi}_L'' + \ubar{f}(\ubar{\phi}_L) = 0, \quad \ubar{\phi}_L'(0) = \varrho^{-1} \left[\ubar{\phi}_L(0) + \eps\right], \quad \ubar{\phi}_L(L) = -\eps.
\end{equation}
Finally, Proposition~\ref{prop:strip-TW} yields a monotone traveling wave $\ubar{\Phi}_L$ and speed $\ubar{c}_L$ satisfying
\begin{equation}
  \label{eq:eps-strip-TW}
  \begin{cases}
    \Delta \ubar{\Phi}_L + \ubar{c}_L \partial_x \ubar{\Phi}_L + \ubar{f}(\ubar{\Phi}_L) = 0 & \textrm{in } \R \times (0, L),\\
    \partial_y\ubar{\Phi}_L(x, 0) = \varrho{-1}\left[\ubar{\Phi}_L(x, 0) + \eps\right],\\
    \ubar{\Phi}_L(x, L) = -\eps,\\
    \ubar{\Phi}_L(-\infty, y) = \ubar{\varphi}_L(y),\\
    \ubar{\Phi}_L(+\infty, y) = -\eps.
  \end{cases}
\end{equation}
By Lemma~\ref{lem:TW-unique}, the speed is unique, as is the wave up to translation.
To determine $\ubar{\Phi}_L$, we fix
\begin{equation}
  \label{eq:eps-pin}
  \ubar{\Phi}_L\left(0, \frac{L}{2}\right) = \theta_0.
\end{equation}

In applying our results to $\ubar{f}$, we must assume that $L$ is sufficiently large.
For instance, Lemma~\ref{lem:steady-ODE-ign-bi} classifies the solutions of \eqref{eq:eps-strip-ODE} provided $L$ exceeds some constant $\ubar{A}_{\trm{ODE}}$.
In general, we let $\ubar{A}_\cdot$ denote the analogue of the various lower bounds $A_\cdot$.
It is easy to see that we can let $\ubar{A}_\cdot$ vary continuously in $\eps$.
In particular, taking $\eps \leq 1$, we can assume that the constants $\ubar{A}_\cdot$ are bounded uniformly in $\eps$.
In fact, the wave and speed also vary continuously in $\eps$:
\begin{lemma}
  \label{lem:eps-TW-cont}
  Let $f$ be ignition with $\varrho \in [0, \infty)$ or bistable with $\varrho = 0$.
  Fix $L > \ubar{A}_{\trm{TW}}$.
  Then $\ubar{\varphi} \to \varphi$ and $\ubar{\varphi}_L \to \varphi_L$ uniformly as $\eps \searrow 0$.
  Moreover, $\ubar{c}_L \to c_L$ and $\ubar{\Phi}_L \to \Phi_L$ uniformly in $\R \times [0, L]$.
\end{lemma}

\begin{proof}
  By standard ODE stability results, $\ubar{\varphi} \to \varphi$ and $\ubar{\varphi}_L \to \varphi_L$ uniformly.
  This leaves the wave and speed.

  Various parameters in our lemmas are continuous in $\eps$, including $\gamma$ and $c_*$.
  Hence the sequence $(\ubar{c}_L)_{\eps  \in (0, 1]}$ is contained in a compact subset of $(0, \infty)$.
  By \eqref{eq:eps-pin}, we can thus extract a subsequence (suppressed for clarity) of $\eps \searrow 0$ such that $\ubar{\Phi}_L$ converges locally uniformly and $\ubar{c}_L$ converges.
  Let $\Phi_L^0$ and $c_L^0 \in [\gamma, c_*]$ denote the corresponding limits.

  The monotone traveling wave $\Phi_L^0$ must converge to solutions of \eqref{eq:strip-ODE} as ${x \to \pm \infty}$.
  Our normalization \eqref{eq:eps-pin} implies that $\Phi_L^0(-\infty, \anon) = \varphi_L$.
  We can handle the right limit as in the proof of Lemma~\ref{prop:strip-TW}.
  That is, let $\ubar{\ti f}$ denote a modification of $\ubar{f}$ which is cut off slightly above $\vartheta.$
  By adjusting the cutoff, we can arrange for the corresponding one-dimensional speed $\ti{\ubar{c}}$ to be small, uniformly in $\eps$.
  Moreover, the one-dimensional wave $\ti{\ubar{U}}$ will decay to $-\eps$ uniformly in $\eps$ as $x \to + \infty$.
  We can again show that there exists a half-strip $[B, \infty) \times [0, L]$  independent of $\eps$ on which $\ubar{\Phi}_L \leq \ti{\ubar{U}}$.
  Taking $\eps \to 0$, the uniformity in $\eps$ of the decay of $\ti{\ubar{U}}$ implies that $\Phi_L^0(+\infty, \anon) = 0$.
  So $\Phi_L^0$ is a solution to \eqref{eq:strip-TW}.

  By Lemma~\ref{lem:TW-unique}, the monotone traveling wave with these limits is unique.
  Hence $\Phi_L^0 = \Phi_L$ and $c_L^0 = c_L$.
  Since the subsequential limit is unique, the entire sequence $(\ubar{\Phi}_L, \ubar{c}_L)_{\eps \in (0, 1]}$ converges to $(\Phi_L, c_L)$.
  Furthermore, the monotonicity of $\Phi_L$ and $\ubar{\Phi}_L$ in $x$ implies that the limit $\ubar{\Phi}_L \to \Phi_L$ is actually uniform in $\R \times [0, L]$.
\end{proof}

\section{Ignition and bistable spreading}
\label{sec:ASP-ign-bi}

We now return to the half-space $\hp \subset \R^{d + 1}$.
We study the solution $u$ of \eqref{eq:main} evolving from compactly supported initial data $u_0$, and prove Theorem~\ref{thm:ASP}\ref{item:ASP-ign-bi}.
We break the proof into several parts.

\subsection{The upper bound}

Using linear and one-dimensional theory, we can quickly establish the upper bound \eqref{eq:upper-ASP}.
In fact, the same argument applies to all classes of reactions and all boundary conditions.
\begin{proposition}
  \label{prop:upper}
  Let $f$ be monostable, ignition, or bistable with $\varrho \in [0, \infty]$.
  Let $u$ solve \eqref{eq:main} with $0 \leq u_0 \leq 1$ compactly supported.
  Then $u$ satisfies \eqref{eq:upper-ASP}.
\end{proposition}

\begin{proof}
  Since we only need an upper bound, we can assume that $\varrho = \infty$, which corresponds to Neumann conditions.
  Then the evolution in $\hp$ agrees with a free evolution in $\R^{d + 1}$ from an initial condition which is even in $y$.
  Let $w_0$ denote the even extension of $u_0$ to $\R^{d + 1}$, and let $w \colon [0,\infty) \times \R^{d + 1} \to \R$ satisfy
  \begin{equation}
    \label{eq:super}
    \begin{cases}
      \partial_t w = \Delta w + f(w),\\
      w(0, \tbf{x}) = w_0(\tbf{x}).
    \end{cases}
  \end{equation}
  Then by our observations above, $w = u$ on $\bar{\hp}$.
  Also, by the strong maximum principle, $w(t, \anon) < 1$ when $t > 0$.
  
  Now define
  \begin{equation*}
    \mu \coloneqq \sup_{s \in (0, 1)} \frac{f(s)}{s}.
  \end{equation*}
  Then the solution to
  \begin{equation*}
    \begin{cases}
      \partial_t W = \Delta W + \mu W,\\
      W(0, \tbf{x}) = w_0(\tbf{x})
    \end{cases}
  \end{equation*}
  is a supersolution to \eqref{eq:super}, so $w \leq W$.
  By an explicit computation with the heat kernel, $W$ decays like a Gaussian as $\abs{x} \to \infty$ at any fixed positive time.
 
  On the other hand, the one-dimensional wave $U$ merely decays exponentially at $+\infty$, and satisfies $U(-\infty) = 1$.
  Since $w(1, \anon) < 1$ decays super-exponentially, there exists a shift $B$ of $U$ such that
  \begin{equation*}
    w(1, \tbf{x}) \leq U(\tbf{e} \cdot \tbf{x} - c_* - B)
  \end{equation*}
  for all $\tbf{e} \in S^{d}$.
  Using $U(\tbf{e} \cdot \tbf{x} -  c_*t - B)$ as a supersolution for each $\tbf{e} \in S^d$, we see that $w$, and hence $u$, cannot propagate in any direction faster than $c_*$.
\end{proof}

\subsection{The lower bound on a slab}

The proof of the lower bound \eqref{eq:upper-ASP} in Theorem~\ref{thm:ASP}\ref{item:ASP-ign-bi} is more involved.
In this subsection, we establish it on a slab of bounded width near $\partial \hp$.
\begin{proposition}
  \label{prop:lower-slab-ign-bi}
  Let $f$ be ignition with $\varrho \in [0, \infty)$ or bistable with $\varrho = 0$.
  Suppose that $(\theta + \delta) \tbf{1}_B \leq u_0 \leq 1$ for some $\delta \in (0, 1 - \theta)$ and some ball $B \subset \hp$ of radius $R_{\trm{steady}}(\delta) > 0$.
  Then for all $\ell > 0$ and $c \in [0, c_*)$,
  \begin{equation}
    \label{eq:lower-ASP-slab}
    \limsup_{t \to \infty} \sup_{\substack{\abs{\tbf{x}'} \leq ct \\ 0 \leq y \leq \ell}} \abs{u(t, \tbf{x}', y) - \varphi(y)} = 0.
  \end{equation}
\end{proposition}

\begin{proof}
  Fix $0 \leq c < c' < c_*$, $\ell > 0$, and $\eta > 0$.
  We use the strip traveling waves $(\Phi_L, c_L)$ from Proposition~\ref{prop:strip-TW}.
  By Lemmas~\ref{lem:strip-ODE-ign-bi} and \ref{lem:speed-convergence}, there exists ${L > \max\{A_{\trm{TW}}, 2\ell\}}$ such that $c_L > c'$ and
  \begin{equation}
    \label{eq:long-convergence-ign-bi}
    \varphi - \frac{\eta}{3} < \varphi_L \leq \varphi \quad \textrm{on } [0, \ell].
  \end{equation}

  We now use the $\eps$-modification $\ubar{f} \leq f$ introduced in Section~\ref{sec:strip-TW}.
  We claim that $\ubar{\varphi}_L < \varphi$ when $\eps$ is sufficiently small.
  By Lemmas~\ref{lem:strip-ODE-ign-bi} and \ref{lem:eps-TW-cont}, $\varphi_L < \varphi$ in $(0, L]$ and $\ubar{\varphi}_L \to \varphi_L$ uniformly as $\eps \to 0$.
  If $\varrho > 0$, we in fact have $\varphi_L < \varphi$ in $[0, L]$, and there is a uniform gap between the two by compactness.
  If follows that $\ubar{\varphi}_L < \varphi$ once $\eps$ is sufficiently small.

  Now suppose $\varrho = 0$.
  Then the ordering is only in doubt when $y \lesssim \eps$, where $\varphi_L$ and $\varphi$ become close.
  If $f$ is ignition, we have $\ubar{f} = f$, and ODE stability results imply that the $-\eps$ boundary condition for $\ubar{\varphi}_L$ reduces $\varphi_L$ by order $\eps$ in a fixed neighborhood of $y = 0$.
  Thus $\ubar{\varphi}_L < \varphi$.
  For bistable reactions, we recall that $\|\ubar{f} - f\|_\infty < \eps^2$.
  Hence the change of reaction only changes $\ubar{\varphi}_L$ by order $\eps^2 \ll \eps$.
  Therefore the boundary condition beats the adjustment of $f$, and we still have $\ubar{\varphi}_L < \varphi.$
  
  By Lemma~\ref{lem:eps-TW-cont}, there exists $\eps > 0$ such that $\ubar{\varphi}_L < \varphi$, $\ubar{c}_L > c'$, and
  \begin{equation}
    \label{eq:steady-strip-distance-ign-bi}
    \big\|\ubar{\varphi}_L - \varphi_L\big\|_{L^\infty([0, L])} < \frac{\eta}{3}.
  \end{equation}
  
  Now define
  \begin{equation}
    \label{eq:large-radius-ign-bi}
    \ubar{R} \coloneqq \frac{d - 1}{\ubar{c}_{L} - c'}
  \end{equation}
  and let $\ubar B$ denote the $d$-dimensional ball of radius $\ubar{R}$ centered at the origin in $\R^d$.
  By Lemma~\ref{lem:invasion-ign-bi}, $u(t, \anon) \to \varphi$ locally uniformly as $t \to \infty$.
  We claim that there exists a time $T > 0$ such that
  \begin{equation}
    \label{eq:comparison-boundary-ign-bi}
    u(t, \tbf{x}', y) \geq \ubar\varphi_{L}(y) \quad \textrm{for all } (t, \tbf{x}', y) \in [T, \infty) \times \bar{\ubar B} \times [0, L].
  \end{equation}
  After all, $\ubar{\varphi}_L < \varphi$ in the compact interval $[0, L]$, so there is a uniform gap between $\varphi$ and $\ubar\varphi_L$.
  Hence the locally uniform convergence $u(t, \anon) \to \varphi$ implies \eqref{eq:comparison-boundary-ign-bi} for sufficiently large $T$.

  We now define a radial subsolution on the region
  \begin{equation*}
    \m A \coloneqq \big\{(t, \tbf{x}', y) \mid t > T, \, \ubar{R} < |\tbf{x}'| < \ubar{R} + c'(t - T), \, 0 < y < L\big\}.
  \end{equation*}
  Since $\ubar{\Phi}_L(+\infty, \anon) = -\eps$, there exists $\ubar{x} > 0$ such that
  \begin{equation*}
    \sup_{y \in [0, L]} \ubar\Phi_L(\ubar{x}, y) = 0.
  \end{equation*}
  We then define
  \begin{equation*}
    w(t, \tbf{x}', y) \coloneqq \ubar\Phi_L\big(|\tbf{x}'| - c'(t - T) - \ubar{R} + \ubar{x}, y\big) \quad \textrm{for } (t, \tbf{x}', y) \in \m A.
  \end{equation*}
  Using the spherical representation of the Laplacian in $\R^d$, we find
  \begin{equation*}
    \partial_t w - \Delta w - f(w) \leq \partial_t w - \Delta w - \ubar{f}(w) = \Big(\ubar c_L - c' - \frac{d - 1}{r}\Big) \partial_r w,
  \end{equation*}
  where $r \coloneqq \abs{\tbf{x}'}$.
  Now $\partial_x \ubar \Phi_L \leq 0$, so $\partial_r w \leq 0$.
  Furthermore, our choice of $\ubar{R}$ in \eqref{eq:large-radius-ign-bi} implies that
  \begin{equation*}
    \ubar c_L - c' - \frac{d - 1}{r} \geq  \ubar c_L - c' - \frac{d - 1}{\ubar{R}} = 0.
  \end{equation*}
  Therefore $\partial_t w - \Delta w - f(w) \leq 0$ in $\m A$, as desired.

  We wish to apply the comparison principle to conclude that $u \geq w$ in $\m A$.
  We must thus check various boundary conditions.
  For fixed $t \geq T$, let $\m A_t$ denote the $t$ time-slice of $\m A$.
  Then $\m A_t$ is an annular cylinder with inner radius $\ubar{R}$, outer radius $\ubar{R} + c'(t - T)$, and height $L$.
  Its boundary has four pieces, corresponding to $\abs{\tbf{x}'} = \ubar{R}$, $\abs{\tbf{x}'} = \ubar{R} + c'(t -T)$, $y = 0$, and $y = L$.

  When $\abs{\tbf{x}'} = \ubar{R}$, \eqref{eq:comparison-boundary-ign-bi} implies that $u \geq \ubar \varphi_L \geq w$.
  When $\abs{\tbf{x}'} = \ubar{R} + c'(t -T)$ or $y = L$, $u \geq 0 \geq w$.
  Finally, $u$ satisfies a ``larger'' boundary condition than $w$  when $y = 0$.
  Since $\m A_T = \emptyset$, there is no initial condition to check.
  Therefore, the comparison principle implies $u \geq w$ in $\m A$.
  
  Finally, consider points in $\m A$ with $\abs{\tbf{x}'} \leq c t$.
  Then
  \begin{equation*}
    |\tbf{x}'| - c'(t - T) - \ubar{R} + \ubar{x} \to -\infty \quad \textrm{as } t \to \infty, 
  \end{equation*}
  so in the definition of $w$ we are evaluating $\ubar \Phi_L$ on the far left.
  Since $\ubar{\Phi}_L(-\infty, \anon) = \ubar{\varphi}_L$, there exists $C > 0$ such that $\ubar \varphi_L - \ubar \Phi_L(x, \anon) < \frac \eta 3$ for all $x \leq -C$.
  Thus, there exists $T' \geq T$ such that
  \begin{equation*}
    u(t, \tbf{x}', y) > \ubar \varphi_L(y) - \frac \eta 3
  \end{equation*}
  for all $t \geq T'$, $\ubar{R} \leq \abs{\tbf{x}'} \leq ct$, and $y \in [0, L]$.
  Furthermore, \eqref{eq:comparison-boundary-ign-bi} allows us to extend this bound to $\abs{\tbf{x}'} \leq \ubar{R}$.
  By \eqref{eq:long-convergence-ign-bi} and \eqref{eq:steady-strip-distance-ign-bi}, we obtain
  \begin{equation*}
    \limsup_{t \to \infty} \sup_{\substack{\abs{\tbf{x}'} \leq c t \\ 0 \leq y \leq \ell}} \big[\varphi(y) - u(t, \tbf{x}', y)\big] < \eta.
  \end{equation*}
  But $\eta > 0$ was arbitrary, so in fact
  \begin{equation}
    \label{eq:liminf-slab-ign-bi}
    \limsup_{t \to \infty} \sup_{\substack{\abs{\tbf{x}'} \leq c t \\ 0 \leq y \leq \ell}} \big[\varphi(y) - u(t, \tbf{x}', y)\big] \leq 0.
  \end{equation}
  
  To control $u - \varphi$ from above, consider the solution $u^1$ to \eqref{eq:main} with $u_0 \equiv 1$.
  It depends only on $y$ and decreases locally uniformly to $\varphi$.
  In fact, $\varphi \leq u^1 \leq 1$ and $\varphi(+\infty) = 1$ imply that $u^1(t, \anon) \to \varphi$ \emph{uniformly}.
  Since $u \leq u^1$, we find
  \begin{equation*}
    \limsup_{t \to \infty} \sup_{\hp} [u(t, \anon) - \varphi] \leq 0.
  \end{equation*}
  In light of \eqref{eq:liminf-slab-ign-bi}, we obtain \eqref{eq:lower-ASP-slab}.
\end{proof}

\subsection{The full lower bound}

With control on slabs, we can extend the propagation into the interior of $\hp$.
The following argument applies to all reactions.
\begin{proposition}
  \label{prop:ASP-interior}
  Let $f$ be monostable or ignition with $\varrho \in [0, \infty)$ or bistable with $\varrho = 0$.
  Let $u$ solve \eqref{eq:main} with $u_0 \leq 1$, and suppose that \eqref{eq:lower-ASP-slab} holds for all $\ell > 0$ and $c \in [0, c_*)$.
  Then $u$ satisfies \eqref{eq:lower-ASP}.
\end{proposition}

\begin{proof}
  Fix $0 \leq c < c' < c_*$ and $\eta > 0$.
  Given $\eps > 0$, we define a new reaction $f_1 \leq \ubar f \leq f$.
  We set $f_1 = \ubar{f}$ on $[-\eps, 1 - 2 \eps]$, and let $f_1$ smoothly connect to $0$ at $1 - \eps$ while remaining below $\ubar{f}$.
  Then $f_1$ is an ignition or bistable reaction on the interval $[-\eps, 1 - \eps]$ with unique one-dimensional speed $c_1$.
  As $\eps \searrow 0$, it is well known that $c_1 \to c_*$.
  Thus, there exists $\eps \in (0, \eta)$ such that $c_1 > c'$.
  Let $U_1$ denote the unique one-dimensional monotone traveling wave for $f_1$ connecting $1 - \eps$ to $-\eps$ such that $U_1(0) = 0$.
  
  Now fix $R_1 \coloneqq \frac{d}{c_1 - c'}$.
  As in the proof of Proposition~\ref{prop:lower-slab-ign-bi}, we can check that
  \begin{equation*}
    w(t, \tbf{x}) \coloneqq U_1\big(\abs{\tbf{x}} - c't\big)
  \end{equation*}
  is a subsolution to the equation $\partial_t w = \Delta w + f(w)$ on $\R^{d + 1} \setminus B_{R_1}$, where $B_{R_1}$ denotes the $(d+1)$-dimensional ball of radius $R_1$ centered at the origin.
  We use $w$ to push $u$ toward $1 - \eps$ out to radius $ct$.

  Since $\varphi(+\infty) = 1$, there exists $\ell \geq R_1$ such that $1 - \varphi \leq \frac \eps 2$ in $[\ell, \infty)$.
  Also, by hypothesis, $u$ satisfies \eqref{eq:lower-ASP-slab} with $c'$ in place of $c$.
  Hence, there exists $T_1 > 0$ such that
  \begin{equation}
    \label{eq:line-bound-ign-bi}
    \inf_{\abs{\tbf{x}'} \leq c't} u(t, \tbf{x}', \ell) \geq 1 - \eps  \quad \text{for all } t \geq T_1.
  \end{equation}
  Now let $\hp_\ell \coloneqq \{(\tbf{x}', y) \in \R^{d + 1} \mid y > \ell\}$ denote the $\ell$-shifted half-space.
  Since $w \leq 1 - \eps$, \eqref{eq:line-bound-ign-bi} ensures that $w \leq u$ on $\partial \hp_\ell \cap B_{c't}$ when $t \geq T_1$.
  Moreover, $w \leq 0 \leq u$ on $\partial \hp_\ell \cap (\R^{d + 1} \setminus B_{c't})$, so in fact $w \leq u$ on the entire boundary $\partial \hp_\ell$ when $t \geq T_1$.
  Also, $w(0, \anon) \leq 0 \leq u(T_1, \anon)$ in $\hp$.
  Thus by the comparison principle,
  \begin{equation*}
    u(t + T_1, \tbf{x}) \geq w(t, \tbf{x}) \quad \text{for all } (t, \tbf{x}) \in [0, \infty) \times \hp_\ell.
  \end{equation*}

  Finally, we are interested in $\tbf{x}$ such that $\abs{\tbf{x}} \leq ct$.
  For such $\tbf{x}$,
  \begin{equation*}
    w(t, \tbf{x}) \geq U_1(-(c' - c)t) \to 1 - \eps \quad \textrm{as } t \to \infty.
  \end{equation*}
  Since $\eps < \eta$ and $1 - \eps \leq \varphi \leq 1$ in this region, there exists $T_2 > T_1$ such that
  \begin{equation}
    \label{eq:upper-circle-ign-bi}
    \sup_{t \geq T_2} \sup_{\substack{\abs{(\tbf{x}', y)} \leq ct \\ y \geq \ell}} \big[\varphi(y) - u(t, \tbf{x}', y)\big] \leq \eta.
  \end{equation}
  In fact, this can be extended to $y \in [0, \ell]$, since $u$ satisfies \eqref{eq:lower-ASP-slab}.
  An upper bound follows as in the proof of Proposition~\ref{prop:lower-slab-ign-bi}.
  Since $\eta > 0$ in \eqref{eq:upper-circle-ign-bi} was arbitrary, we obtain \eqref{eq:lower-ASP}.
\end{proof}

Finally, Theorem~\ref{thm:ASP}\ref{item:ASP-ign-bi} follows from Lemma~\ref{lem:invasion-ign-bi} and Propositions~\ref{prop:upper}, \ref{prop:lower-slab-ign-bi}, and \ref{prop:ASP-interior}.

\section{Ignition and bistable traveling waves}
\label{sec:TW-ign-bi}

We are now in a position to construct traveling waves for ignition and bistable reactions.
As noted in the introduction, it suffices to consider waves in two spatial dimensions, so we assume $d = 1$ and denote position by $\tbf{x} = (x, y) \in \R^2$.

\begin{proof}[Proof of Theorem~\textup{\ref{thm:TW}\ref{item:TW-ign-bi}}]
  First, let $f$ be ignition or bistable with $\varrho = 0$.
  Define
  \begin{equation*}
    \theta_1 \coloneqq \frac{1}{2}\varphi(1) \in (0, 1).
  \end{equation*}
  Recall $A_{\trm{TW}}$ and the traveling waves $(\Phi_L, c_L)$ from Proposition~\ref{prop:strip-TW}.
  By Lemma~\ref{lem:strip-ODE-ign-bi}, there exists $A \geq A_{\trm{ODE}}$ such that $\varphi_L(1) > \theta_1$ for all $L > A$.
  Then define $x_L$ by
  \begin{equation}
    \label{eq:pin-near-boundary}
    \Phi_L(x_L, 1) = \theta_1,
  \end{equation}
  which exists because
  \begin{equation*}
    \Phi_L(-\infty, 1) = \varphi_L(1) > \theta_1 > 0 = \Phi_L(+\infty, 1).
  \end{equation*}
  We consider the sequence $\big(\Phi_L( \anon + x_L, \anon), c_L\big)_{L > A}$ as $L \to \infty$.
  We have \eqref{eq:pin-near-boundary} and $c_L \in [\gamma, c_*]$.
  It follows from elliptic estimates and Lemma~\ref{lem:speed-convergence} that there exists a locally uniform subsequential limit $(\Phi, c_*)$ satisfying
  \begin{equation*}
    \begin{cases}
      \Delta \Phi + c_* \partial_x \Phi + f(\Phi) = 0 & \textrm{in } \hp,\\
      \Phi = 0 & \textrm{on } \partial \hp.
    \end{cases}
  \end{equation*}
  Furthermore, $0 \leq \Phi \leq 1$, $\partial_x \Phi \leq 0$, and $\Phi(0, 1) = \theta_1$.
  Since $\partial_y \Phi_L > 0$ on $\R \times (0, L/2)$, we also have $\partial_y \Phi \geq 0$.

  Now, the monotonicity in $x$ ensures that the limits $\Phi(\pm\infty, \anon)$ exist and satisfy the ODE \eqref{eq:steady-ODE}.
  Furthermore, the limits are bounded and satisfy
  \begin{equation}
    \label{eq:limit-separation}
    \Phi(-\infty, 1) \geq \frac 1 2 \varphi(1) \geq \Phi(+\infty, 1).
  \end{equation}
  By Lemma~\ref{lem:steady-ODE-ign-bi}, the only bounded solutions to \eqref{eq:steady-ODE} are $0$ and $\varphi$.
  By \eqref{eq:limit-separation}, we must have
  \begin{equation*}
    \Phi(-\infty, \anon) = \varphi \And \Phi(+\infty, \anon) = 0.
  \end{equation*}
  Therefore, $\Phi$ is a traveling wave of speed $c_*$.
  Moreover, the strong maximum principle implies that $0 < \Phi < 1$, $\partial_x \Phi < 0,$ and $\partial_y \Phi > 0$ in $\hp$.

  If $f$ is ignition and $\varrho > 0$, we instead use the symmetric Robin waves $(\Phi_L^{\mathrm{sym}}, c_L^{\mathrm{sym}})$ from Section~\ref{sec:strip-TW}.
  By Proposition~\ref{prop:omnibus-TW-ign-sym}, the argument goes through as above.
  This establishes existence in Theorem~\ref{thm:TW}\ref{item:TW-ign-bi} at the speed $c_*$.

  For nonexistence at slower speeds, we appeal to Theorem~\ref{thm:ASP}\ref{item:ASP-ign-bi}.
  Suppose for the sake of contradiction that $\Psi$ is a traveling wave of speed $c \in [0, c_*)$, and fix $c' \in (c, c_*)$.
  Since $\Psi(-\infty, \anon) = \varphi$ and $\varphi(+\infty) = 1$, the wave is arbitrarily close to $1$ on arbitrarily large balls.
  Hence, given $\delta \in (0, 1 - \theta)$, there exists a ball $B \subset \hp$ of the radius $R_{\trm{steady}}(\delta) > 0$ from Theorem~\ref{thm:steady}\ref{item:steady-ign-bi} such that $\Psi \geq (\theta + \delta) \tbf{1}_B$.
  Let $u$ solve \eqref{eq:main} with $u_0 = (\theta + \delta) \tbf{1}_B$.
  Then Theorem~\ref{thm:ASP}\ref{item:ASP-ign-bi} shows that $u$ eventually approaches $\varphi$ in balls of radius $c't$.
  So $u$ overtakes the traveling wave solution $\Psi(x - ct, y)$, which connects to $0$ and travels slower.
  That is, there exists $(T, x, y) \in [0, \infty) \times \hp$ such that $u(T, x, y) > \Psi(x - cT, y)$, contradicting the comparison principle.
  Thus, there do not exist waves slower than $c_*$.
\end{proof}
This concludes our analysis of ignition and bistable reactions.

\section{Monostable steady states}
\label{sec:steady-mono}

In the remaining sections, we study monostable reactions.
We begin by proving the existence and uniqueness of a nonzero bounded steady state in $\R_+$.
\begin{lemma}
  \label{lem:steady-ODE-mono}
  Let $f$ be monostable with $\varrho \in [0, \infty)$.
  Then there exists a unique nonzero bounded solution $\varphi$ to the ODE \eqref{eq:steady-ODE}.
  Furthermore, $\varphi$ satisfies $0 \leq \varphi < 1$, $\varphi' > 0$, and $\varphi(+\infty) = 1$.
\end{lemma}

\begin{proof}
  Suppose $\varphi$ is a nonzero bounded solution.
  Since $f$ vanishes outside $[0, 1]$, $\varphi$ becomes affine linear if it exits this interval.
  Then $\abs{\varphi}$ would grow without bound, so necessarily $\varphi([0, \infty]) \subset [0, 1]$.
  As a consequence of the boundary condition, we obtain $\varphi'(0) > 0$.
  By \ref{hyp:mono}, $\varphi$ is concave in $\R_+$.
  Hence, $\varphi$ increases towards a zero of $f$.
  This zero can only be $1$, so $0\leq \varphi < 1$, $\varphi'>0$, and $\varphi(+\infty) = 1$.
  
  Existence and uniqueness follow as in the proof of Lemma~\ref{lem:steady-ODE-ign-bi}.
\end{proof}

We now use the sliding method to extend uniqueness to the half-space ${\hp \subset \R^{d + 1}}$.
\begin{proof}[Proof of Theorem~\textup{\ref{thm:steady}}\ref{item:steady-mono}]
  The one-dimensional solution $\varphi$ from Lemma~\ref{lem:steady-ODE-mono} is a nonzero bounded steady state, so we need only establish uniqueness. 

  Let $\phi$ be some nonzero bounded solution of the steady state equation
  \begin{equation}
    \label{eq:steady}
    \begin{cases}
      \Delta \phi + f(\phi) = 0 & \textrm{in } \hp,\\
      \partial_y \phi = \varrho^{-1} \phi & \textrm{on } \partial \hp.
    \end{cases}
  \end{equation}
  Arguing as in the proof of Proposition~\ref{prop:steady-ign-bi-sub}, the comparison principle implies that $0 \leq \phi \leq \varphi$.

  We now define
  \begin{equation*}
    \rho \coloneqq \inf_{s \in (0, \, 1/2]} \frac{f(s)}{s}.
  \end{equation*}
  By \ref{hyp:mono} and \ref{hyp:mono-derivs}, $\rho > 0$.
  Therefore, there exists $R > 0$ such that $\rho$ is the principal Dirichlet eigenvalue of $-\Delta$ in the $(d + 1)$-dimensional ball $B$ of radius $R$.
  Let $v$ denote the corresponding positive eigenfunction, normalized by $\norm{v}_\infty = 1$.
  We extend $v$ by $0$ outside $B$, and define its translation $\tau_{\tbf{x}} v \coloneqq v(\anon - \tbf{x})$ for
  \begin{equation*}
    \tbf{x} \in \hp_R \coloneqq \big\{(\tbf{x}', y) \in \R^{d + 1} \mid y > R\big\}.
  \end{equation*}

  By the strong maximum principle, $\phi > 0$ in $\hp$.
  Thus, there exist $\tbf{x}_0 \in \hp_R$ and $\eps \in (0, 1/2]$ such that $\eps \tau_{\tbf{x}_0} v \leq \phi$.
  By the normalization of $v$ and the definition of $\rho$, $\eps \tau_{\tbf{x}_0} v$ is a subsolution to \eqref{eq:steady}.
  Hence, $\eps \tau_{\tbf{x}_0} v < \phi$ by the strong maximum principle.
  If we slide $\tbf{x}_0$ around within $\hp_R$, the strong maximum principle further implies that $\eps \tau_{\tbf{x}_0} v$ cannot touch $\phi$ from below.
  So in fact
  \begin{equation*}
    V \coloneqq \eps \sup_{\tbf{x}_0 \in \hp_R} \tau_{\tbf{x}_0} v \leq \phi.
  \end{equation*}
  Note that $V$ is independent of $\tbf{x}'$ and positive in $\hp$.
  Furthermore, as a supremum of subsolutions, $V$ itself is a subsolution to \eqref{eq:steady}.

  Again, $V$ evolves under \eqref{eq:main} towards a bounded solution of \eqref{eq:steady-ODE}.
  But as a subsolution, $V$ increases under the evolution \eqref{eq:main}.
  Since $V > 0$, its long-time limit cannot be $0$.
  By Lemma~\ref{lem:steady-ODE-mono}, this limit is $\varphi$.
  Thus by the comparison principle, $\varphi \leq \phi \leq \varphi$.
\end{proof}

\section{Monostable spreading}
\label{sec:ASP-mono}

We now consider the evolution of $u$ under \eqref{eq:main} from compactly supported initial data.
By Proposition~\ref{prop:upper}, we need only prove the lower bound \eqref{eq:lower-ASP}.
We begin by establishing the hair-trigger effect.
\begin{lemma}
  \label{lem:invasion-mono}
  Let $f$ be monostable with $\varrho \in [0, \infty).$
  If $0 \lneqq u_0 \leq 1$, then
  \begin{equation*}
    \lim_{t \to \infty} u(t, \anon) = \varphi
  \end{equation*}
  locally uniformly in $\hp$.
\end{lemma}

\begin{proof}
  In the proof of Theorem~\ref{thm:steady}\ref{item:steady-mono}, we used the principal Dirichlet Laplacian eigenfunction $v$ on a ball of radius $R$.
  By construction, $\eps \tau_{\tbf{x}_0}v$ is a subsolution to \eqref{eq:main} for all $\eps \in [0, 1/2]$ and $\tbf{x}_0 \in \hp_{R}$.
  
  By the Harnack inequality, $u(1, \anon) > 0$ in $\hp$.
  Thus it lies above $\eps \tau_{\tbf{x}_0} v$ for some $\eps > 0$ and $\tbf{x}_0 \in \hp_{R}$.
  If we evolve \eqref{eq:main} from initial data $\eps \tau_{\tbf{x}_0} v$ and $1$, we sandwich $u$ between solutions which converge locally uniformly to nonzero bounded steady states.
  By Theorem~\ref{thm:steady}\ref{item:steady-mono}, $\varphi$ is the unique such state.
  Thus, $u(t, \anon) \to \varphi$ locally uniformly as $t \to \infty$.
\end{proof}

We wish to upgrade this convergence to the quantitative lower bound \eqref{eq:lower-ASP}.
By Proposition~\ref{prop:ASP-interior}, it suffices to show convergence on a slab near $\partial \hp$.
We again use a ``reduced'' reaction $\ubar{f}$ and its corresponding steady states and traveling waves.
We let $\ubar{f} \coloneqq f$, but we view the reaction on the larger interval $[-\eps, 1]$, where it is ignition.
\begin{lemma}
  \label{lem:eps-steady-comparison}
  For each $\eps > 0$, $\ubar{\varphi} < \varphi$.
\end{lemma}

\begin{proof}
  We first note that $\varphi$ is increasing in the Robin parameter $\varrho$.
  Indeed, if ${\varrho^{(1)} < \varrho^{(2)}}$ correspond to solutions $\varphi^{(1)}$ and $\varphi^{(2)}$ of \eqref{eq:steady-ODE}, respectively, then $\varphi^{(1)}$ is a subsolution to the $\varrho^{(2)}$-equation.
  By Lemma~\ref{lem:steady-ODE-mono}, $\varphi^{(2)}$ is the unique nonzero bounded solution to the $\varrho^{(2)}$-equation.
  It follows that $\varphi^{(2)} > \varphi^{(1)}$.

  Now suppose that $\ubar{\varphi}(0) > 0$.
  Then we can write
  \begin{equation*}
    \ubar \varphi'(0) = \varrho^{-1}\frac{\ubar\varphi(0) + \eps}{\ubar\varphi(0)} \ubar\varphi(0) \eqqcolon \ubar\varrho^{-1} \ubar \varphi(0)
  \end{equation*}
  for some $\ubar \varrho < \varrho$.
  That is, $\ubar{\varphi}$ is the steady state for a smaller Robin parameter $\ubar{\varrho}$.
  Since $\varphi$ is increasing in $\varrho$, $\varphi > \ubar{\varphi}$.
  Taking $\ubar{\varrho} \searrow 0$, the same holds when $\ubar{\varphi}(0) = 0$.

  We are left with the case $\ubar{\varphi}(0) < 0$.
  Then there exists $\ubar{y} > 0$ such that $\ubar{\varphi}(\ubar{y}) = 0$.
  Since $\ubar{f} = f$, $\ubar{\varphi}(\anon + \ubar{y})$ is a Dirichlet solution to \eqref{eq:steady-ODE}.
  By the ordering in $\varrho$ and monotonicity in $y,$
  \begin{equation*}
    \varphi \geq \ubar\varphi(\anon + \ubar{y}) > \ubar{\varphi}.
    \qedhere
  \end{equation*}
\end{proof}

We can now show propagation near $\partial \hp$.
\begin{proposition}
  \label{prop:lower-slab-mono}
  Let $u$ solve \eqref{eq:main} with $0 \lneqq u_0 \leq 1$.
  Then for all $\ell > 0$ and $c \in [0, c_*)$, $u$ satisfies \eqref{eq:lower-ASP-slab}.
\end{proposition}

\begin{proof}
  Fix $0 \leq c < c' < c_*$, $\ell > 0$, and $\eta > 0$.
  By ODE stability, $\ubar \varphi \to \varphi$ as $\eps \to 0$.
  Recall that $\ubar{c}$ denotes the speed of the one-dimensional wave for $\ubar{f}$ connecting $1$ to $-\eps$.
  By standard results from reaction-diffusion theory, $\ubar c \to c_*$ as $\eps \to 0$.
  We can thus choose $\eps > 0$ such that $\ubar c > c'$ and
  \begin{equation}
    \label{eq:steady-distance}
    \big\|\varphi - \ubar \varphi\big\|_\infty < \frac{\eta}{3}.
  \end{equation}

  Now recall the monotone traveling wave $\big(\ubar{\Phi}_L, \ubar{c}_L\big)$ satisfying \eqref{eq:eps-strip-TW} for $L > \ubar{A}_{\trm{TW}}$.
  By Lemmas~\ref{lem:TW-unique} and \ref{lem:speed-convergence}, it is unique up to translation and $\ubar{c}_L \to \ubar{c}$ as $L \to \infty$.
  Using the uniform convergence in Lemma~\ref{lem:strip-ODE-ign-bi}, we can thus choose $L > \max\big\{\ubar{A}_{\trm{TW}}, 2 \ell\big\}$ such that $\ubar{c}_L > c'$ and
  \begin{equation}
    \label{eq:wide-distance}
    \big\|\ubar \varphi_L - \ubar \varphi\big\|_{L^\infty([0, \, \ell])} < \frac{\eta}{3}.
  \end{equation}
  Moreover, Lemmas~\ref{lem:strip-ODE-ign-bi} and \ref{lem:eps-steady-comparison} imply
  \begin{equation}
    \label{eq:steady-comparison}
    \ubar{\varphi}_L < \ubar{\varphi} < \varphi.
  \end{equation}

  Using Lemma~\ref{lem:invasion-mono} and \eqref{eq:steady-comparison}, we may now proceed as in the proof of Proposition~\ref{prop:lower-slab-ign-bi}.
  We can construct a radial subsolution from $\ubar{\Phi}_L$ moving outward at speed $c'$, and use it to push $u$ up close to $\ubar{\varphi}_L$.
  In particular, there exists $T$ such that
  \begin{equation}
    \label{eq:sub-push}
    u(t, \tbf{x}', y) \geq \ubar{\varphi}_L(y) - \frac \eta 3
  \end{equation}
  for all $t \geq T$, $\abs{\tbf{x}'} \leq ct$, and $y \in [0, L]$.
  For details, see the proof of Proposition~\ref{prop:lower-slab-ign-bi}.
  
  Combining \eqref{eq:steady-distance}, \eqref{eq:wide-distance}, and \eqref{eq:sub-push}, we find
  \begin{equation*}
    \limsup_{t \to \infty} \sup_{\substack{\abs{\tbf{x}'} \leq c t \\ 0 \leq y \leq \ell}} \big[\varphi(y) - u(t, \tbf{x}', y)\big] < \eta.
  \end{equation*}
  Recalling that $\eta > 0$ was arbitrary, we in fact have
  \begin{equation}
    \label{eq:liminf-slab-mono}
    \limsup_{t \to \infty} \sup_{\substack{\abs{\tbf{x}'} \leq c t \\ 0 \leq y \leq \ell}} \big[\varphi(y) - u(t, \tbf{x}', y)\big] \leq 0.
  \end{equation}
  
  To control $u - \varphi$ from above, we use the uniform convergence of $1$ to $\varphi$ under the evolution \eqref{eq:main}, as in the proof of Proposition~\ref{prop:lower-slab-ign-bi}.
  This implies
  \begin{equation*}
    \limsup_{t \to \infty} \sup_{\hp} [u(t, \anon) - \varphi] \leq 0.
  \end{equation*}
  In combination with \eqref{eq:liminf-slab-mono}, we obtain \eqref{eq:lower-ASP-slab}.
\end{proof}

Now, Theorem~\ref{thm:ASP}\ref{item:ASP-mono} follows from Propositions~\ref{prop:upper}, \ref{prop:lower-slab-mono}, and \ref{prop:ASP-interior}.

\section{Monostable traveling waves}
\label{sec:TW-mono}

Finally, we construct monostable traveling waves.
Throughout, we assume that $f$ is monostable, $\varrho \in [0, \infty)$, and $d = 1$.
With the results of the preceding section, we can immediately prove half of Theorem~\ref{thm:TW}\ref{item:TW-mono}.
\begin{proof}[Proof of nonexistence in Theorem~\textup{\ref{thm:TW}\ref{item:TW-mono}}]
  Suppose for the sake of contradiction that $\Psi$ is a traveling wave of speed $c \in [0, c_*)$, and take $c' \in (c, c_*)$.
  Since $\Psi > 0$, ${\Psi \geq \delta \tbf{1}_B}$ for some $\delta \in (0, 1)$ and some nonempty open ball $B \subset \hp$.
  By Theorem~\ref{thm:ASP}\ref{item:ASP-mono}, the solution $u$ to \eqref{eq:main} beginning from $u_0 = \delta \tbf{1}_B$ eventually approaches $\varphi$ in balls of radius $c't$.
  This necessarily overtakes the traveling wave solution $\Psi(x - ct, y)$, contradicting the comparison principle.
  So $\Psi$ does not exist.
\end{proof}

Now fix $c \geq c_*$.
We follow the approach of \cite{BRR3} to construct a traveling wave $\Phi$ of speed $c$.
By definition, $\Phi$ is a steady solution to \eqref{eq:main} in the $c$-moving frame:
\begin{equation}
  \label{eq:moving-steady}
  \begin{cases}
    \Delta \Phi + c \partial_x \Phi + f(\Phi) = 0 & \textrm{on } \hp,\\
    \partial_y \Phi = \varrho^{-1} \Phi & \textrm{on } \partial \hp.
  \end{cases} 
\end{equation}
Moreover, $\Phi$ satisfies \eqref{eq:connection}.

We first construct a family of subsolutions.
As in the proof of Theorem~\ref{thm:steady}\ref{item:steady-mono}, define
\begin{equation*}
  \rho \coloneqq \inf_{s \in (0, 1/2]} \frac{f(s)}{s} > 0
\end{equation*}
as well as
\begin{equation*}
  \ell_0 \coloneqq \frac{1}{\sqrt{\rho}} \arctan\big(\varrho \sqrt{\rho}\big) \And \ell_1 \coloneqq \frac{\pi}{2\sqrt{\rho}} - \ell_0.
\end{equation*}
Then let
\begin{equation*}
  v(y) \coloneqq
  \begin{cases}
    \sin\left[\sqrt{\rho}(y + \ell_0)\right] & \textrm{for } 0 \leq y \leq \ell_1,\\
    1 & \textrm{for } y \geq \ell_1.
  \end{cases}
\end{equation*}
By construction, $v$ is $\m C^1$ and satisfies the absorbing boundary condition at $y = 0$.
Moreover, $k v$ is a subsolution to \eqref{eq:moving-steady} for all $k \in [0, 1/2]$.
As a consequence,
\begin{equation}
  \label{eq:strict-subsolution}
  \frac{1}{2} v < \varphi \quad \textrm{in } \R_+.
\end{equation}
Indeed, $\frac{1}{2} v$ is a subsolution to \eqref{eq:steady-ODE}.
Since it is not a \emph{solution}, its evolution under the parabolic version of \eqref{eq:steady-ODE} is strictly increasing in time.
Moreover, $\frac{1}{2} v$ lies under the bounded supersolution $1$, so its parabolic evolution converges to a nonzero bounded steady state.
By Lemma~\ref{lem:steady-ODE-mono}, this state is $\varphi$.
Since the evolution of $\frac{1}{2} v$ is strictly increasing in time, \eqref{eq:strict-subsolution} follows.

Next, we need a corresponding supersolution.
\begin{lemma}
  \label{lem:mono-super}
  There exists a supersolution $\Psi$ of \eqref{eq:moving-steady} with the following properties: $\partial_y \Psi = \varrho^{-1} \Psi$ on $\partial \hp$, $\partial_x \Psi < 0$ and $\partial_y \Psi > 0$ in $\hp$, $\Psi(+\infty, \anon) = 0$, and there exists $B \in \R$ such that
  \begin{equation}
    \label{eq:super-bound}
    \Psi > \frac 1 2 v \quad \textrm{on } (-\infty, B] \times [0, \infty).
  \end{equation}
\end{lemma}

\begin{proof}
  Since $f$ is monostable, there exists a one-dimensional wave $U^c$ of speed $c$ connecting $1$ to $0$.
  Let $u$ solve the parabolic form of \eqref{eq:moving-steady} with $u_0 = U^c$, and let
  \begin{equation*}
    \Psi(x, y) \coloneqq u(1, x, y).
  \end{equation*}
  Since $\partial_x u_0 < 0$, we have $\partial_x \Psi < 0$.
  Similarly, $u_0$ is independent of $y$, so the absorbing boundary condition and the strong maximum principle yield $\partial_y \Psi > 0$.
  Now, $u_0$ is a supersolution of \eqref{eq:moving-steady}, so $u$ is decreasing in $t$.
  Since $U^c(+\infty) = 0$, it follows that $\Psi(+\infty, \anon) = 0$.

  We must now grapple with the behavior of $\Psi$ on the far left.
  We recall that ${U^c(-\infty) = 1}$.
  By parabolic regularity, it follows that $u(t, -\infty, y)$ solves the one-dimensional parabolic problem
  \begin{equation}
    \label{eq:half-line-parabolic}
    \begin{cases}
      \partial_t \omega = \partial_y^2 \omega + f(\omega) & \textrm{in } \R_+,\\
      \partial_y \omega(t, 0) = \varrho^{-1} \omega(t, 0),\\
      \omega(0, y) = 1.
    \end{cases}
  \end{equation}
  That is, $\Psi(-\infty, y) = \omega(1, y)$.
  By Dini's theorem, the convergence
  \begin{equation*}
    \Psi(x, y) \to \Psi(-\infty, y)
  \end{equation*}
  is locally uniform in $y$.
  In fact, boundedness and monotonicity in $y$ imply that the convergence is \emph{uniform}.

  Next, note that $\frac{1}{2} v$ is a subsolution of \eqref{eq:half-line-parabolic}.
  Since $\frac{1}{2}v < 1 = \omega(0, \anon)$, we have $\Psi(-\infty, y) \geq \frac{1}{2} v(y)$.
  Suppose $\varrho > 0$.
  Then Hopf and the strong maximum principle imply strict inequality: $\Psi(-\infty, \anon) > \frac{1}{2}v$.
  Since $\frac{1}{2} v(+\infty) = \frac{1}{2}$, the uniform convergence of $\Psi$ on the left implies the existence of $B \in \R$ such that \eqref{eq:super-bound} holds.
  If $\varrho = 0$, we must contend with the behavior of our limit near $y = 0$.
  The Hopf lemma implies that
  \begin{equation*}
    \partial_y \Psi(-\infty, 0) > \frac{1}{2} v'(0).
  \end{equation*}
  Moreover, parabolic regularity implies that the limit $x \to -\infty$ commutes with $\partial_y$.
  Hence, the convergence $\Psi(x, y) \to \Psi(-\infty, y)$ holds in ${\m C}_y^1$.
  It follows that there exists $\eps > 0$ and $B_0 \in \R$ such that
  \begin{equation*}
    \Psi > \frac 1 2 v \quad \textrm{on } (-\infty, B_0] \times [0, \eps].
  \end{equation*}
  With the behavior near the boundary taken care of, we can argue as before to produce $B$ satisfying \eqref{eq:super-bound}.
\end{proof}
\noindent
For each $h \in \R$, let $\Psi^h(x, y) \coloneqq \Psi(x + h, y)$ denote the leftward shift of $\Psi$ by $h$.

Given $a, b > 0$, we define the bounded box
\begin{equation*}
  \Omega_{ab} \coloneqq (-a, a) \times (0, b)
\end{equation*}
and the multiple
\begin{equation*}
  k \coloneqq \min\left\{\inf_{y \in (0, b)} \frac{\Psi^h(a, y)}{v(y)}, \, \frac 1 2\right\}.
\end{equation*}
Then $k \in (0, 1/2]$ by the strong maximum principle and the Hopf lemma.
Hence $kv$ is a subsolution and $kv \leq \Psi^h$ in $\Omega_{ab}$ (since $\Psi$ is decreasing in $x$).
We use this ordered pair of sub- and supersolutions to construct a solution to
\begin{equation}
  \label{eq:TW-box-mono}
  \begin{cases}
    \Delta \Phi_\sq + c \partial_x \Phi_\sq + f(\Phi_\sq) = 0 & \textrm{on } \Omega_{ab},\\
    \partial_y \Phi_\sq = \varrho^{-1} \Phi_\sq & \textrm{on } \partial \Omega_{ab} \cap \partial \hp,\\
    \Phi_\sq(x, y) = \frac{a - x}{2a} \Psi^h(x, y) + \frac{a + x}{2a} k v(y) & \textrm{on } \partial \Omega_{ab} \cap \hp.\\
  \end{cases}
\end{equation}

\begin{lemma}
  There exists a unique solution to \eqref{eq:TW-box-mono} satisfying $k v \leq \Phi_\sq \leq \Psi^h$.
  Furthermore, $\partial_x \Phi_\sq < 0$ and $\partial_y \Phi_\sq > 0$ in $\Omega_{ab}$.
\end{lemma}

\begin{proof}
  By construction, $k v \leq \Psi^h$ are sub- and supersolution to \eqref{eq:TW-box-mono}, respectively.
  Thus the parabolic evolution of \eqref{eq:TW-box-mono} from either $k v$ or $\Psi^h$ will be monotone in time, and will converge to a solution $\Phi_\sq$ between $k v$ and $\Psi^h$ as $t \to \infty$.

  We took some care in the construction of $v$ and $\Psi$ so that they satisfy the absorbing boundary condition at $y = 0$.
  This ensures that $\Phi_\sq$ is $\m C^1$.
  With this regularity, uniqueness and monotonicity follow from the sliding arguments in the proof of Lemma~\ref{lem:box-ign-bi}. 
\end{proof}

We use the solutions $\Phi_\sq$ as approximations of traveling waves in the half-plane.
\begin{proof}[Proof of existence in Theorem~\textup{\ref{thm:TW}}.]
  We exploit the dependence of $\Phi_\sq$ on $h$, the translation of $\Psi^h$.
  We therefore write $k^h$ and $\Phi_\sq^h$ for clarity.
  (Note that $\Phi_\sq^h$ is \emph{not} simply a shift of $\Phi_\sq$.)
  Recall $\ell_1$ from the construction of $v.$
  For each $a, b > \ell_1$, Lemma~\ref{lem:mono-super} implies the existence of $\ubar h \in \R$ such that $k^h = \frac 1 2$ for all $h \leq \ubar h$.
  It follows that
  \begin{equation*}
    \Phi_\sq^{\ubar h}(0, \ell_1) > \frac 1 2 v(\ell_1) = \frac 1 2.
  \end{equation*}
  Furthermore, standard elliptic estimates show that $\Phi_\sq^h$ is continuous in $h$.
  Since $\Psi(+\infty, \anon) = 0$, it follows that there exists $h_* > \ubar h$ such that $k^{h_*} < \frac 1 2$ and ${\Phi_\sq^{h_*}(0, \ell_1) = \frac 1 2}$.
  
  For each $a, b > \ell_1$, we have selected a shift $h_*(a, b) \in \R$.
  We now take $a, b \to \infty$.
  By elliptic regularity, $\Phi_\sq^{h_*}$ converges locally uniformly along a subsequence to a solution $\Phi$ to \eqref{eq:moving-steady} satisfying $\partial_x \Phi \leq 0$, $\partial_y \Phi \geq 0$, $0 \leq \Phi \leq 1$, and $\Phi(0, \ell_1) = \frac 1 2$.

  We must now verify the limiting behavior \eqref{eq:connection}.
  Monotonicity in $x$ implies that the limits $\Phi(\pm \infty, \anon)$ exist and satisfy
  \begin{equation*}
    \Phi(-\infty, \ell_1) \geq \Phi(0, \ell_1) = \frac{1}{2} \geq \Phi(+\infty, \ell_1).
  \end{equation*}
  In light of \eqref{eq:strict-subsolution}, this implies that
  \begin{equation}
    \label{eq:limit-inequalities}
    \Phi(-\infty, \ell_1) > 0 \And \Phi(+\infty, \ell_1) < \varphi(\ell_1).
  \end{equation}
  On the other hand, elliptic estimates show that the limits $\Phi(\pm \infty, \anon)$ are bounded solutions of \eqref{eq:steady-ODE}.
  By Lemma~\ref{lem:steady-ODE-mono}, the only such solutions are $0$ and $\varphi$.
  From \eqref{eq:limit-inequalities}, we obtain \eqref{eq:connection}.
  Therefore, $\Phi$ is a traveling wave.
  The strong maximum principle and the Hopf lemma now imply that $0 < \Phi < 1$, $\partial_x \Phi < 0$, and $\partial_y \Phi > 0$ in $\hp$.
\end{proof}

\printbibliography

\end{document}